\documentclass[12pt, a4paper, oneside]{article}
\usepackage[utf8]{inputenc}
\usepackage[UKenglish]{babel}
\usepackage[toc,page]{appendix}
\usepackage{csquotes}
\usepackage{amsmath,mathtools}
\usepackage{amssymb}
\usepackage{amsthm}
\usepackage{amsfonts}
\usepackage{graphicx}
\usepackage{wasysym}
\usepackage[left=2.4cm, right=2.4cm, top=2.5cm, bottom=3cm,]{geometry}
\usepackage{textgreek}
\usepackage{float}
\usepackage{subcaption}
\usepackage{xfrac}
\usepackage[dvipsnames]{xcolor}

\usepackage{bm}

\usepackage{tikz}
    \usetikzlibrary{cd, positioning, calc, fit}
    \usetikzlibrary{decorations.pathreplacing,calligraphy}
    \tikzstyle{every picture}+=[remember picture]
\usepackage{enumitem}
\usepackage[many]{tcolorbox}
\usepackage[backend=biber,
            maxbibnames=99,
            natbib=true,
            sorting=anyt,
            style=numeric,
            citestyle=numeric-comp,
            url=false,
            eprint=false,
            isbn=false,
            doi=false,
            autocite=inline]{biblatex}
\usepackage{hyperref}
    \hypersetup{
    colorlinks=true,
    linkcolor=black,
    filecolor=magenta,
    urlcolor=black,
    citecolor=blue,
    linktocpage=true,
    breaklinks=true,
    unicode=true,
    psdextra,
    }
\usepackage{enumitem}
    \setlist{nosep}
\usepackage{tabularx}
\usepackage{multirow}
\usepackage{tablefootnote}

\usepackage{array}
\newcolumntype{L}[1]{>{\raggedright\let\newline\\
                        \arraybackslash\hspace{0pt}}m{#1}}
\newcolumntype{C}[1]{>{\centering\let\newline\\
                        \arraybackslash\hspace{0pt}}m{#1}}
\newcolumntype{R}[1]{>{\raggedleft\let\newline\\
                        \arraybackslash\hspace{0pt}}m{#1}}

\usepackage{algpseudocode, algorithm, algorithmicx}

\newtheorem{theorem}{Theorem}[section]

\newtheorem{lemma}[theorem]{Lemma}

\newtheorem{thm}{Theorem}[section]
\newtheorem{prop}[thm]{Proposition}
\newtheorem{defi}[thm]{Definition}
\newtheorem{cor}[thm]{Corollary}
\newtheorem{meth}[thm]{Method}

\theoremstyle{definition}

\newtheorem{example}[theorem]{Example}

\theoremstyle{remark}
\newtheorem{remark}[theorem]{Remark}

\tcolorboxenvironment{quote}{
    blanker,
    breakable,
    left=5mm,
    before skip=10pt,
    after skip=10pt,
    borderline west={1mm}{0pt}{black!5!white}
}

\makeatletter
\newcommand*{\bigcdot}{}
\DeclareRobustCommand*{\bigcdot}{%
  \mathbin{\mathpalette\bigcdot@{}}%
}
\newcommand*{\bigcdot@scalefactor}{.7}
\newcommand*{\bigcdot@widthfactor}{1.15}
\newcommand*{\bigcdot@}[2]{%
  \sbox0{$#1\vcenter{}$}
  \sbox2{$#1\cdot\m@th$}%
  \hbox to \bigcdot@widthfactor\wd2{%
    \hfil
    \raise\ht0\hbox{%
      \scalebox{\bigcdot@scalefactor}{%
        \lower\ht0\hbox{$#1\bullet\m@th$}%
      }%
    }%
    \hfil
  }%
}
\makeatother

\newtcbox{\mybox}[1][blue]{  on line,
                            arc=0pt,
                            outer arc=0pt,
                            colback=#1!10!white,
                            colframe=#1!50!black,
                            boxsep=0pt,
                            left=0pt,
                            right=0pt,
                            top=0.5pt,
                            bottom=0.5pt,
                            boxrule=0pt,
                            }

\newcommand{\R}{\mathbb{R}}
\newcommand{\N}{\mathbb{N}}
\newcommand{\tE}{\tilde{\mathbb{E}}}

\newcommand{\id}{\,\mathrm{d}}

\newcommand{\der}[0]{\mathrm{d}}
\newcommand{\E}[1]{\mathbb{E}\left[ #1  \right]}
\newcommand{\Pas}{$\mathbb{P}$-a.s.}
\newcommand{\vecdiv}[0]{\operatorname{\overrightarrow{\operatorname{div}}}}
\newcommand{\grad}[0]{\operatorname{grad}}

\DeclareMathOperator*{\Tr}{Tr}
\DeclareMathOperator{\Hess}{Hess}
\DeclareMathOperator{\diver}{div}

\usepackage{setspace}
\usepackage{tocloft}
 
\title{An application of the splitting-up method for the computation of a neural network representation for the solution for the filtering equations\footnote{Submitted to \emph{Stochastics and Partial Differential Equations: Analysis and Computations},
special volume on the occasion of Istvan Gyongy's 70th birthday.}}

\author{Dan Crisan$^{1}$, 
        Alexander Lobbe$^{1,2}$, 
        Salvador Ortiz-Latorre$^2$
        \\
        \\
        \footnotesize{$^1$Department of Mathematics,
        Imperial College London, UK}\\ 
        \footnotesize{\href{mailto:d.crisan@imperial.ac.uk}{d.crisan@imperial.ac.uk}}\\
        \footnotesize{\href{mailto:alex.lobbe@imperial.ac.uk}{alex.lobbe@imperial.ac.uk}}\\[0.5cm]
        \footnotesize{$^2$Department of Mathematics,
        University of Oslo, Norway}\\ 
        \footnotesize{\href{mailto:salvadoo@math.uio.no}{salvadoo@math.uio.no}} 
    }

\date{}

\usepackage{fancyhdr}
\allowdisplaybreaks

\newdimen\figrasterwd
\figrasterwd\textwidth

\begin{document}

\captionsetup{width=.9\linewidth}

\maketitle

\begin{abstract}
    \noindent
    The filtering equations govern the evolution of the conditional distribution of a signal process given partial, and possibly noisy, observations arriving sequentially in time.
    Their numerical approximation plays a central role in many real-life applications, including numerical weather prediction \cite{Llopis_filt_nwp, Galanis20062451}, finance \cite{brigo_filt_fin, Date_filt_fin}\cite[Part IX]{handbook} and engineering \cite{filt_eng}.
    One of the classical approaches to approximate the solution of the filtering equations is to use a PDE inspired method, called the splitting-up method, initiated by Gyongy, Krylov, LeGland, among other contributors, see e.g., \cite{gyongy_03, legland_92}.
    This method, and other PDE based approaches, have particular applicability for solving low-dimensional problems.
    In this work we combine this method with a neural network representation inspired by \cite{Han8505}. The new methodology is used to produce an approximation of the unnormalised conditional distribution of the signal process. We further develop a recursive normalisation procedure to recover the normalised conditional distribution of the signal process. The new scheme can be iterated over multiple time steps whilst keeping its asymptotic unbiasedness property intact.
    We test the neural network approximations with numerical approximation results for the Kalman and Benes filter.
\end{abstract}
\newpage
\hrulefill
\vspace{-0.5em}
{\footnotesize
\tableofcontents
}
\hrulefill\vspace{-1.5em}
\begin{flushright}    
{\footnotesize Date: \today}
\end{flushright}\vspace{-2em}
\section{Introduction}
This paper is concerned with the numerical approximation of the solution of the stochastic filtering equations.
In addition to its theoretical significance in stochastic analysis and control (see, for example, \cite{bain_crisan_2008}, \cite{handbook} or \cite{bensoussan_book}), stochastic filtering is an important modelling framework for many domains of application, such as numerical weather prediction \cite{Llopis_filt_nwp, Galanis20062451}, finance \cite{brigo_filt_fin, Date_filt_fin}\cite[Part IX]{handbook} and engineering \cite{filt_eng}.
Hence, there is a high demand for efficient and accurate numerical methods to approximate the solution of the filtering problem, i.e. the solution of the filtering equations.
Here, we are presenting a first study in an ongoing effort to combine a machine learning approach, that has risen in prominence within the numerical community over the past years, with the classical PDE based approach to the numerical resolution of the stochastic filtering problem.
In particular, we base our algorithm on the SPDE splitting method that was, among others, developed by Istvan Gyongy, Nikolay Krylov and Francois LeGland \cite{gyongy_03, legland_92}. The chosen neural network based machine learning approach for the approximation of the involved deterministic PDE is inspired by \cite{Han8505}.

Among all contributors, Istvan Gyongy has made the most fundamental contribution to the development of the splitting-up method as applied to the filtering equation and beyond. In the following we give some brief details of his contribution to the topic. The first of Gyongy's works in this direction was 
published in 2002 \cite{gyongy_spdeappl} where he presented numerical results for the approximation of stochastic PDEs with a particular focus on the the splitting-up method. Soon after, he published the paper \cite{gyongy_2003} with Nikolay Krylov, in the Annals of Probability. In this work, he investigates the convergence rates of the splitting method for various different classes of stochastic PDEs. Furthermore, in the final part of the paper he explicitly treats the application of these results in the context of stochastic filtering. In another work with Krylov in 2003, Gyongy proved convergence rates in Sobolev norm for the splitting-up method. Notably, this result is proved for the general case of time-dependent coefficients of the considered classes of SPDEs and the rates are even shown to be sharp. A short while later, another work of Gyongy, coauthored by Krylov, appeared in the year 2005 \cite{gyongy_split_extrap}. In this innovative paper, Gyongy devised a theoretical method for the splitting-up approximation of parabolic equations by constructing high order splitting-up methods out of low order ones by means of Richardson extrapolation.

The paper is structured as follows:
In subsection~\ref{sec:notation} we introduce the notation in the paper. Thereafter, in subsection~\ref{sec:filtering-problem}, we present the stochastic filtering problem at the level of generality appropriate for the purposes of this work. Notably, Proposition~\ref{prop:KS} presents the well-known Kallianpur-Striebel formula which establishes the distinction between what we call, respectively, the normalised and unnormalised filter. Subsequently, in subsection~\ref{sec:F_eq}, we discuss the filtering equations and recall the splitting-up method as we will apply it to the stochastic filtering equations. Based on the SPDE for the unnormalised filter, sometimes referred to as Zakai's equation, we apply the splitting method to decompose the SPDE into the deterministic PDE part and a normalisation, or data-assimilation, step.
The first step is commonly solved numerically by using Galerkin methods or similar grid-based approximation schemes. This approach is best applied in low-dimensional settings, due to the computational cost introduced by the discretisation. The second step is to construct the (approximate) likelihood based on the observation and to finally normalise the product of the likelihood function and the PDE solution such that it integrates to unity.

Next, in section~\ref{sec:representatio}, we analyse the case when the coefficient functions of the differential operator in the deterministic PDE that arises from the splitting-up method has smooth coefficients. The consequence of this assumption is that the operator can be split into a diffusion operator and a zero-order part. An elementary but crucial part of our argument is then given in Lemma~\ref{prop:diffusion} which establishes the fact that the diffusion operator arising from the PDE operator with smooth coefficients generates a stochastic diffusion process, which we will later call \emph{auxiliary diffusion}. Another central ingredient in the derivation of our method is the Feynman-Kac formula, given in Theorem~\ref{thm:finalValuePDE} for final-value PDEs. As we are presented with an initial-value problem, we will need the Feynman-Kac formula in a form that applies to such kind of PDEs. This is given in Corollary~\ref{thm:initialValuePDE}. The significance of the Feynman-Kac formula and the auxiliary diffusion derived in section~\ref{sec:representatio} lies in the fact that the solution to the deterministic PDE problem can then be written as a conditional expectation with respect to the law of the auxiliary diffusion given its initial value. We then give two examples of explicit representations of solutions to the particular filtering problems of the Kalman (linear) filter and the Benes filter in terms of the Feynman-Kac representation. In subsection~\ref{sec:optim} we prove Proposition~\ref{thm:minimizer} based on arguments presented in \cite{Han8505} and thus show that the solution of the PDE over a full hypercube-domain is represented by an infinite-dimensional optimisation of an objective function given by the Feynman-Kac formula.

Section~\ref{sec:DL_method} is dedicated to the detailed description of our computational method. In subsection~\ref{sec:NN_model} we introduce some terminology on deep learning, and specify how a parametrised neural network representation of the solution of the deterministic PDE is approximated through a Monte-Carlo sampling-based minimisation of the objective function given by the Feynman-Kac formula and the minimisation problem derived before.
In practise, the infinite-dimensional function space over which we theoretically minimise is parametrised by the neural network parameters to make it computationally tractable. This enables us to use generic methods for the computational optimisation, provided we are able to sample from the auxiliary diffusion process.
Thereafter, in subsection~\ref{sec:mc_correction} we describe the second part of the splitting method where we rely on the Monte-Carlo approximation
of the product of the neural network and the likelihood function to obtain the necessary normalisation constant.
Subsequently, in subsection~\ref{sec:algorithm} we describe the neural network representation and the chosen optimisation algorithm mathematically which results in a full specification of our method in terms of pseudocode. In particular, the algorithm may be iterated over several time steps whilst remaining asymptotically unbiased.

The numerical results obtained for the Kalman and Benes filters are presented in section~\ref{sec:num_res}. In our one-dimensional examples, we observed that the method can successfully be iterated over several time-steps. To the best of our knowledge, we present the first numerical results showing that the sampling based neural network representation of the solution to the Fokker-Planck equation may be iterated while remaining accurate with respect to the exact solution of the filtering problem. In fact, the filtering framework is ideally suited for this kind of study, because of its inherently sequential nature. Moreover, we identify the choice of the domain as a crucial factor for the success of our approximation. Due to the normalisation procedure which uses samples from the likelihood, we need to have a good signal-to-noise ratio in order to obtain a large proportion of samples within our considered domain. If this is not the case, the method diverges. Our study of the nonlinear Benes filter shows that the method is able to handle also nonlinear dynamics. 

In conclusion, based on the limited testing performed in this study, we believe that the use of neural network based representations in the numerical approximation of the stochastic filtering problem can be a viable alternative to existing numerical methods. Nevertheless, we emphasize the following two important caveats. First, the mathematical analysis of deep learning algorithms such as the one we employed here is not advanced enough to guarantee explicit convergence rates which might be undesirable in certain settings. Secondly, more numerical studies have to be performed to accurately evaluate the capabilities of neural networks in situations of higher practical relevance than the synthetic study we have performed in this work.
We plan to investigate this topic further in future work.

\subsection{Notation}\label{sec:notation}
Throughout this paper, $\mathbb{N}$ denotes the natural numbers without zero
and $\mathbb{N}_0$ is the set of natural numbers including zero.
The real numbers are denoted by $\R$ and, given $n\in \mathbb{N}$,
$\R^n$ is $n$-dimensional Euclidean space.
For $m,n\in\mathbb{N}$ the set of
$m\times n$-matrices with real entries is denoted by $\R^{m\times n}$.
For a given matrix $M\in\R^{m\times n}$, $M^\prime$ denotes its transpose
and $\Tr(M)$ denotes its trace.
For $k\in \mathbb{N}_0\cup\{\infty\}$ and separable normed
$\R$-vector spaces $A$ and $B$ we denote by $C^k(A,B)$ the set of $k$-times
continuously differentiable functions from $A\rightarrow B$. Moreover,
we use the shorthand $C^k(A)$ whenever $A=B$ and always identify 
$C^0(A,B)=C(A,B)$. 
Similarly, the spaces of 
$k$-times continuously differentiable functions with compact support
are denoted by $C_c^k(A,B)$ and the ones of bounded functions
with bounded derivatives of all orders by $C_b^k(A,B)$.
For an interval $I\subset\R$ and $d\in\N$ we write $C_b^{1,2}(I\times\R^d,\R)$ for the set
of bounded functions $f:I\times\R^d\rightarrow\R$ that are once continuously 
differentiable with bounded derivative in the first variable and twice 
continuously differentiable with bounded derivative in the second variable.
For a topological space $(\mathcal{T},\mathfrak{T})$,
$\mathcal{B(\mathcal{T})}$ is the Borel sigma-algebra on $\mathcal{T}$.
Further, if $(\mathcal{T}_0,\mathfrak{T}_0)$ is another topological space,
then $B(\mathcal{T}, \mathcal{T}_0)$ denotes the set of
bounded Borel-measurable functions from $\mathcal{T} \rightarrow \mathcal{T}_0$.
For a measurable space $(M,\mathfrak{M})$ we write $\mathcal{P}(M)$ for the
set of probability measures on $(M,\mathfrak{M})$ and $\mathcal{M}(M)$ for the
set of all measures on $(M,\mathfrak{M})$.
For $d\in\N$ and $a \in C^1(\R^d, \R^{d\times d})$ we write
\begin{equation*}
    \vecdiv(a)= \left( \sum_{i=1}^d \partial_i a_{ij} \right)_{j=1}^d.
\end{equation*}
Moreover, when $f\in C^1(\R^d,\R)$ we denote the gradient of $f$ by $\grad f$
and the divergence of $f$ by $\diver f$.
When $g\in C^2(\R^d,\R)$ we denote the Hessian of $g$ by $\Hess g$.

\subsection{Stochastic filtering problem}\label{sec:filtering-problem}
In this section we are following Bain and Crisan \cite{bain_crisan_2008}.
Let $(\Omega, \mathcal{F}, \mathbb{P})$ be a probability space with a normal
filtration\footnotemark $\,(\mathcal{F}_t)_{t\geq0}$.
\footnotetext{
    We call the filtration $(\mathcal{F}_t)_{t\geq0}$ \emph{normal}, if 
\begin{itemize}
    \item[--] $\mathcal{F}_0$ contains all $\mathbb{P}$-nullsets of 
    $\mathcal{F}$, and
    \item[--] $(\mathcal{F}_t)_{t\geq0}$ is right-continuous.
\end{itemize}}
Let $d,p\in\N$ and let $X: [0,\infty) \times \Omega \rightarrow \R^d$ be a
$d$-dimensional stochastic process satisfying, for all $t\in[0,\infty)$ 
and \Pas, that
\begin{equation}\label{eq:signal}
    X_t = X_0 + \int_0^t f(X_s) \id s + \int_0^t \sigma(X_s) \id V_s \;,
\end{equation}
where
$f: \R^d \rightarrow \R^d$ and
$\sigma: \R^d \rightarrow \R^{d \times p}$
are globally Lipschitz continuous functions and
$V: [0,\infty) \times \Omega \rightarrow \R^p$ is a $p$-dimensional 
$(\mathcal{F}_t)_{t\geq0}$-adapted Brownian motion.
Then $X$ admits the infinitesimal generator
$A: \mathcal{D}(A) \rightarrow B(\R^{d})$ given,
for all $\varphi \in \mathcal{D}(A)$, by
\begin{equation}\label{eq:A}
    A\varphi = \langle f, \nabla \varphi \rangle + \Tr(a\Hess\varphi),
\end{equation}
where $\mathcal{D}(A)$ denotes the domain of the differential operator $A$
and where we defined the function
$a(\cdot) = \frac{1}{2}\sigma(\cdot)\sigma^\prime(\cdot) : 
\R^d \rightarrow \R^{d \times d}$.
We assume from now on that a dense core for the domain $\mathcal{D}(A)$ is $C_c^2(\R^d)$.

In the context of stochastic filtering, $X$ is called the 
\emph{signal process}.
Further, we assume the \emph{observation process}
$Y: [0,\infty) \times \Omega \rightarrow \R^{m}$ to be given, 
for all $t\in[0,\infty)$ and \Pas, by
\begin{equation}\label{eq:observation}
    Y_t = \int_0^t h(X_s) \id s + W_t \;,
\end{equation}
where $W: [0,\infty) \times \Omega \rightarrow \R^{m}$ is an
$(\mathcal{F}_t)_{t\geq0}$-adapted Brownian motion independent of $V$.
The \emph{sensor function} $h: \R^{d} \rightarrow \R^{m}$ is a
globally Lipschitz continuous function with the property that 
for all $t\in[0,\infty)$, \Pas,
\begin{equation*}
    \E{\int_0^t h(X_s)^2 \id s} < \infty \;
    \text{ and } \; \E{\int_0^t Z_s h(X_s)^2 \id s} < \infty,
\end{equation*}
where the stochastic process $Z:[0,\infty) \times \Omega \rightarrow \R$
is defined such that for all $t\in[0,\infty)$,
\begin{equation*}
Z_t = \exp\{-\int_0^t h(X_s)\id W_s - \frac{1}{2} \int_0^t h(X_s)^2 \id s\}.
\end{equation*}
We specify the \emph{observation filtration} for $t\geq0$ by
\begin{equation*}
    \mathcal{Y}_t = {\sigma}(Y_s, s\in[0,t]) \vee \mathcal{N}
    \hspace{10pt}\text{ and write }\hspace{10pt} \mathcal{Y} =
    \sigma\left(\bigcup_{t \in [0,\infty)} \mathcal{Y}_t\right),
\end{equation*}
where $\mathcal{N}$ is the collection of $ \mathbb{P}$-nullsets of
$\mathcal{F}$.
Then we are interested in the $(\mathcal{Y}_t)_{t\geq 0}$-adapted stochastic 
process $\pi: [0,\infty) \times \Omega  \rightarrow \mathcal{P}(\R^d)$
that is defined by the requirement that for all
$\varphi \in B(\R^d,\R)$ and $t\in [0,\infty)$ it holds \Pas{} that
\begin{equation*}
    \pi_t\varphi = \E{ \varphi(X_t) \left| \mathcal{Y}_t \right. }.
\end{equation*}
The process $\pi$ is often called the \emph{filter}.
Under this model, the stochastic process $Z$ is an 
$(\mathcal{F}_t)_{t\geq0}$-martingale and by Novikov's condition
we can use Girsanov's theorem to define the change of measure given by 
$\left. \frac{\id \tilde{\mathbb{P}}^t}
{\id \mathbb{P}}\right|_{\mathcal{F}_t} = Z_t$, $t\geq 0$.
Note that on $\bigcup_{t\in [0,\infty)}\mathcal{F}_t$ we have a
consistent measure $\tilde{\mathbb{P}}$ in place of $\tilde{\mathbb{P}}^t$.
Moreover, the signal and observation processes $X$ and $Y$
are independent under the new measure and $Y$ is a Brownian motion under 
$\tilde{\mathbb{P}}$. 
Furthermore, under $\tilde{\mathbb{P}}$, we can define the stochastic process 
$\rho: [0,\infty) \times \Omega  \rightarrow \mathcal{M}(\R^d)$
by the requirement that for all
$\varphi \in B(\R^d,\R)$ and $t\in [0,\infty)$ it holds \Pas{} that
\begin{equation}\label{eq:rho}
    \rho_t\varphi = \E{ \left.\varphi(X_t) 
    \exp\{\int_0^t h(X_s)\id Y_s - \frac{1}{2} \int_0^t h(X_s)^2 \id s\}
    \right| \mathcal{Y}_t }.
\end{equation}
The following important Proposition~\ref{prop:KS}, known
in the literature as the Kallianpur-Striebel formula,
justifies the terminology to call $\rho$ the
\emph{unnormalised filter}.

\begin{prop}[Kallianpur-Striebel formula]\label{prop:KS}
    For all $t\geq 0$ and $\varphi \in B(\R^d,\R)$ 
    it holds $\tilde{\mathbb{P}}$-a.s. that
    \begin{equation*}
    \pi_t(\varphi) =
      \frac{\displaystyle \rho_t(\varphi)}{\displaystyle 
      \rho_t(\mathbf{1})}
    = \frac{\displaystyle \tE\left[ \left. \varphi(X_t)
        \exp\{\int_0^t h(X_s)\id Y_s -
        \frac{1}{2} \int_0^t h(X_s)^2 \id s\} \right| 
        \mathcal{Y}\right]}
        {\displaystyle \tE\left[\left. 
        \exp\{\int_0^t h(X_s)\id Y_s -
        \frac{1}{2} \int_0^t h(X_s)^2 \id s\} 
        \right| \mathcal{Y}\right]},
    \end{equation*}
    where $\mathbf{1}$ is the constant function $\R^d\ni x \mapsto 1$.
\end{prop}
The proof of Proposition~\ref{prop:KS} can be found in, e.g., 
\cite{bain_crisan_2008}.

\subsection{Filtering equation and general splitting method}\label{sec:F_eq}
It is well established in the literature (see, e.g., \cite{bain_crisan_2008}), 
that the unnormalised filter $\rho$, defined in \eqref{eq:rho},
satisfies the \emph{filtering equation}, i.e. for all $t\geq 0$ it holds 
$\tilde{\mathbb{P}}$-a.s. that
\begin{equation}
    \label{eq:Zakai_rho}
    \rho_t(\varphi) = \pi_0(\varphi) +
        \int_0^t \rho_s(A\varphi)\id s +
        \int_0^t \rho_s(\varphi h^\prime) \id Y_s.
\end{equation}
Moreover, it is known (see, e.g., \cite[Theorem 7.8]{bain_crisan_2008})
that if $\pi_0$ is absolutely continuous with respect to Lebesgue measure and
such that it has a square-integrable density, and if additionally the 
sensor function $h$ is uniformly bounded, then $\rho_t$ admits 
a square-integrable density $p_t$ with respect to the Lebesgue measure on 
$\R^d$. 
Then, assuming the necessary regularity for $p_t$ (see, e.g., \cite[Theorem 7.12]{bain_crisan_2008}, for the precise condition), the 
Zakai equation \eqref{eq:Zakai_rho} implies that, for all
$t\geq 0$ and $\varphi \in C_c^\infty(\R^d,\R)$, 
we have $\tilde{\mathbb{P}}$-a.s. that
\begin{equation*}
    \rho_t(\varphi) = \int_{\R^d} \varphi(x) p_t(x) \id x,
\end{equation*}
where $A^*$ is the formal adjoint of the infinitesimal generator
$A$ of the signal process $X$, given by the relation
\begin{equation*}
    \int_{\R^d} A\varphi(x) p_t(x) \id x = \int_{\R^d} \varphi(x) A^*p_t(x)\id x;
    \; t\geq 0.
\end{equation*}

The PDE method we will consider is from~\cite{cai1995adaptive} and
seeks to approximate
the following  stochastic partial differential equation
(SPDE) for the density $p_t$ given, for all $t\geq0$, $x\in\R^d$, 
and \Pas{} as
\begin{equation*}
    p_t(x) = p_0 (x) + \int_0^t A^* p_s(x)\id s + 
    \int_0^t h^\prime(x) p_s(x)  \id Y_s
\end{equation*}
and relies on the splitting-up algorithm described in \cite{le1989time}
and \cite{legland1992splitting}.

Choose a final time $T>0$ and an integer $N\in\N$ and let 
$\{t_0=0<\dots<t_N=T\}$
be a discretisation of the time interval $[0,T]$.
Then the first step of the splitting-up approach,
also called \emph{prediction} step, is to numerically
approximate the \emph{Fokker-Planck equation} 
\begin{equation}\label{eq:FP_PDE}
    \begin{aligned}
    \frac{\partial q}{\partial t} (t,z)  &= A^* q (t,z), 
    &\; & (t,z)\in (0,T]\times\R^d,\\
        q(0,z) &= p_{0}(z),  &\;  & z\in \R^d,
    \end{aligned}
\end{equation}
over the discretised interval.
To this end, note that the first prediction step of the method consists of the
numerical approximation of the solution $q^1$ of the PDE
\begin{equation*}
    \begin{aligned}
    \frac{\partial q^1}{\partial t} (t,z) &= A^* q^1 (t,z), 
    &\; & (t,z)\in (0,t_1]\times\R^d,\\
        q^1(0,z) &= q^0(0,z) := p_{0}(z), &\;   & z\in \R^d.
    \end{aligned}
\end{equation*}
We denote the numerical approximation of $q^1(t_1,\cdot)$ by $\tilde{p}^1$.
Next, we employ the second step of the method, the so-called \emph{correction} 
step, which consists of the normalisation of the obtained Fokker-Planck 
approximations using the observation process $Y$, as given by 
\eqref{eq:observation}, and the Kallianpur-Striebel
formula (see Proposition~\ref{prop:KS}).
To illustrate this, the first correction step is calculated as follows.
Let 
\begin{equation*}
    z_1 = \frac{1}{t_1-t_0}(Y_{t_1}-Y_{t_0}), \;\text{ consider the function }\; 
    \R^d\ni z\mapsto\xi_1(z) = \exp \left( -\frac{1}{2} ||z_1 - h(z)||^2 
    \right),
\end{equation*}
and define for all $z\in\R^d$,
\begin{equation*}
    p^1(z) = C_1 \xi_1(z) \tilde{p}^1(z),
\end{equation*}
where $C_1$ is the normalisation constant such that
$\int_{\R^d} p^1(z) \id z = 1$.

Therefore, we formulate the splitting-up method below in Method~\ref{meth:PDE}.

\begin{meth}\label{meth:PDE}
The full method is defined by iterating the above steps with
${p}^{0}(\cdot)= p_0(\cdot)$ and such that
for all $n\in\{1,\dots,N\}$ we iteratively calculate
\begin{enumerate}[label=\arabic*)]
    \item\label{predictor}
    an approximation $\tilde{p}^n$ of the solution to
    \begin{equation}\label{eq:num_IVP}
        \begin{aligned}
        \frac{\partial q^n}{\partial t} (t,z) &= A^* q^n (t,z), 
        &\; & (t,z)\in (t_{n-1},t_n]\times\R^d,\\
            q^n(0,z) &= {p}^{n-1}(z), &\;    & z\in \R^d,
        \end{aligned}
    \end{equation}
    at time $t_n$ and
    \item
    the normalisation based on
    \begin{equation*}
        z_n = \frac{1}{t_n-t_{n-1}}(Y_{t_n}-Y_{t_{n-1}}) 
    \end{equation*}
    and the function
    \begin{equation*}
        \R^d\ni z\mapsto\xi_n(z) = \exp \left( -\frac{t_n-t_{n-1}}{2} 
        ||z_n - h(z)||^2 \right),
    \end{equation*}
    so that we can define for all $z\in\R^d$,
    \begin{equation*}
        p^n(z) = \frac{1}{C_n} \xi_n(z) \tilde{p}^n(z),
    \end{equation*}
    where $C_n =
    \int_{\R^d} \xi_n(z)\tilde{p}^n(z) \id z $.
\end{enumerate}
\end{meth}
In this article, we replace the predictor step \ref{predictor} in 
Method~\ref{meth:PDE} above by a deep 
neural network approximation algorithm to avoid an explicit space 
discretisation which has exponential complexity in the space dimension $d$.
This will be achieved by representing each $\tilde{p}^n(z)$ by a feed-forward
neural network and approximating the initial value problem \eqref{eq:num_IVP}
based on its stochastic representation using a sampling procedure.


\section{Feynman-Kac representation and auxiliary diffusion}\label{sec:representatio}
In this section we consider the case when the coefficient functions of the
signal and the observation processes are sufficiently smooth and thus allow
the expansion of the partial differential operator $A^*$.
Based on this expansion we can rewrite the Fokker-Planck
equation \eqref{eq:FP_PDE} as a Kolmogorov equation plus, in general, a 
zeroth-order term.
The reason to do so is that the so obtained representation enables the
use of the Feynman-Kac formula (see Theorem~\ref{thm:finalValuePDE} below)
to rewrite the solution of the PDE
problem as an expectation of an appropriately chosen stochastic process.
Thus, we can approximate this expectation by Monte-Carlo sampling from the 
diffusion.

This particular approach follows a recent stream of research into deep learning based
approximations of PDEs which is mainly focused on high dimensional problems, see, e.g.
\cite{2016arXiv161107422H, 2017arXiv170604702E, 2018arXiv180600421B, beck2019deep} and 
related works within the context of stochastic optimal control \cite{pham2019neural, hur2018deep, hur2019machine, pham2019neural, pham2021neural}.
Alternative approaches, typically based on incorporating the PDE directly into the loss function, for the approximation of a neural network representation of solutions of PDEs are also actively developed in the literature, see, for example, \cite{raissi2017physics, Sirignano_2018, alaradi2018solving}. 

\subsection{Fokker-Planck equation}
We begin by expanding the differential operator under the assumption that it has smooth coefficient functions.
    As before, we let $d,p\in\N$,
    $f=(f_i)_{i=1}^d \in C_b^1(\R^d,\R^d)$, 
    let $\sigma=(\sigma_{ij})_{j=1,\dots,p}^{i=1,\dots,d} 
    \in C_b^2(\R^d,\R^{d\times p})$,
    let $a=(a_{ij})_{i,j=1}^{d}\in  C_b^2(\R^d,\R^{d\times d})$
    be the function that maps
    $x \mapsto  \frac{1}{2}\sigma(x)\sigma^\prime(x)$ and let
    $A^*:C_c^\infty(\R^d,\R)\rightarrow C_b(\R^d,\R)$ be the 
    partial differential operator  
    with the property that for all $\varphi \in C_c^\infty(\R^d,\R)$,
    \begin{equation*}
        A^* \varphi = - \sum_{i=1}^d \frac{\partial}{\partial x_i} f_i\varphi
            + \sum_{i,j=1}^d \frac{\partial^2}{\partial x_i \partial x_j}
            a_{ij}\varphi.
    \end{equation*}
    Then, for all $\varphi \in C_c^\infty(\R^d,\R)$ we have
    \begin{equation}
    \label{eq:op_expansion}
        A^* \varphi = \Tr(a\Hess \varphi) +
            \langle 2\vecdiv(a)-f, \grad \varphi \rangle +
            \diver(\vecdiv(a) - f)\varphi.
    \end{equation}

\begin{defi}\label{def:A_hat}
    Let $d,p\in\N$,
    $f=(f_i)_{i=1}^d \in C_b^1(\R^d,\R^d)$, 
    let $\sigma=(\sigma_{ij})_{j=1,\dots,p}^{i=1,\dots,d} 
    \in C_b^2(\R^d,\R^{d\times p})$,
    and let $a=(a_{ij})_{i,j=1}^{d}\in  C_b^2(\R^d,\R^{d\times d})$
    be the function that maps
    $x \mapsto  \frac{1}{2}\sigma(x)\sigma^\prime(x)$.
    Then we define the partial differential operator 
    $\hat{A}:C_c^\infty(\R^d,\R)\rightarrow C_b(\R^d,\R)$ such that
    for all $\varphi \in C_c^\infty(\R^d,\R)$,
    \begin{equation*}
        \hat{A} \varphi = \Tr(a\Hess \varphi) +
            \langle 2\vecdiv(a)-f, \grad \varphi \rangle
    \end{equation*}
    and we define the function $r:\R^d\rightarrow \R$ such that for all
    $x\in \R^d$,
    \begin{equation*}
        r(x) = \diver(\vecdiv(a)-f)(x).
    \end{equation*}
\end{defi}

\begin{lemma}\label{prop:diffusion}
    For all $x\in \R^d$ the operator $\hat{A}$ defined in
    Definition~\ref{def:A_hat} is the infinitesimal generator
    of the Itô diffusion $\hat{X}:[0,\infty)\times \Omega \rightarrow \R^d$
    given, for all $t\geq0$ and \Pas{} by
    \begin{equation*}
        \hat{X}_t =x + \int_0^t b(\hat{X}_s) \der s +
            \int_0^t \sigma(\hat{X}_s) \der \hat{W}_s,
    \end{equation*}
    where $\hat{W}:[0,\infty)\times \Omega \rightarrow \R^d$
    is a $d$-dimensional Brownian motion and
    $b:\R^d \rightarrow \R^d$ is the function
    \begin{equation*}
        b = 2\vecdiv(a) - f.
    \end{equation*}
\end{lemma}
\begin{proof}
    cf. \cite[Chapter IV, Theorem 6.1]{IkedaWatanabe1981}
\end{proof}

The next Theorem~\ref{thm:finalValuePDE} is the well-known Feynman-Kac formula.

\begin{thm}[Feynman-Kac formula]\label{thm:finalValuePDE}
    Let $d\in\N$, $T>0$, $k\in C(\R^d, [0,\infty))$,
    let $\hat{A}$ be the operator defined in
    Definition~\ref{def:A_hat},
    and let $\psi:\R^d \rightarrow \R$ be an at most polynomially growing
    function\footnotemark.
     If $v \in C_b^{1,2}([0,T)\times \R^d,\R)$
    satisfies the Cauchy problem
    \begin{equation}\label{eqn:finalValuePDE}
    \begin{aligned}
        - \frac{\partial v}{\partial t}(t,x) + k(x)v(t,x) &= \hat{A}v(t,x),&\; &
        (t,x)\in[0,T)\times \R^d,\\
        v(T,x) &= \psi(x),  &\; &x\in \R^d,
    \end{aligned}
    \end{equation}
    then we have for all $(t,x) \in [0,T)\times \R^d$ that
    \begin{equation*}
        v(t,x) = \E{\left. \psi(\hat{X}_T)\exp\left(-\int_t^T k(\hat{X}_\tau)
             \id\tau\right) \right| \hat{X}_t = x},
    \end{equation*}
    where $\hat{X}$ is the diffusion generated by $\hat{A}$.
\end{thm}
\footnotetext{
    I.e. there exist real numbers 
    $\lambda \geq 1$ and $L\geq 0$ such that for all $x\in \R^d$,     
    $|\psi(x)| \leq L(1+\| x \|^{2\lambda})$.
    
}
\begin{proof}
See \cite[Chapter 5, Theorem 7.6]{KaratzasShreve1998}.
The assumption that the coefficients of $\hat{A}$ have bounded derivatives
ensures that the required conditions are met.
\end{proof}

From Theorem~\ref{thm:finalValuePDE} above we can deduce the 
Corollary~\ref{thm:initialValuePDE} below about the
initial value problem corresponding to~\eqref{eqn:finalValuePDE}.

\begin{cor}\label{thm:initialValuePDE}
    Under the assumptions of the previous Theorem~\ref{thm:finalValuePDE},
    suppose that $u \in C_b^{1,2}((0,T]\times \R^d,\R)$ satisfies the Cauchy 
    problem
    \begin{equation}\label{eqn:initialValuePDE}
    \begin{aligned}
         \frac{\partial u}{\partial t}(t,x) + k(x)u(t,x) &= \hat{A}u(t,x),  &\; 
         &(t,x)\in (0,T] \times \R^d,\\
        u(0,x) &= \psi(x), &\; &x\in \R^d.
    \end{aligned}
    \end{equation}
    Then, for all $(t,x) \in (0,T]\times \R^d$, we have that
    \begin{equation*}
        u(t,x) = \E{\left. \psi(\hat{X}_t)\exp\left(-\int_0^t k(\hat{X}_\tau)
        \id\tau\right) \right| \hat{X}_0 = x},
    \end{equation*}
    where $\hat{X}$ is the diffusion generated by $\hat{A}$.
\end{cor}

\begin{proof}
    For all $(s,x)\in (0,T]\times \R^d$, set $u(s,x) = v(T-s,x)$,
    where $v \in C_b^{1,2}([0,T)\times \R^d,\R)$ 
    satisfies \eqref{eqn:finalValuePDE}.
    Then, $u \in C_b^{1,2}((0,T]\times \R^d,\R)$ and
    \eqref{eqn:finalValuePDE} implies that 
    $u$ satisfies \eqref{eqn:initialValuePDE}, i.e.
    \begin{equation*}
    \begin{aligned}
         \frac{\partial u}{\partial t}(t,x) + k(x)u(t,x) &= \hat{A}u(t,x),  &\; 
         &(t,x)\in (0,T] \times \R^d,\\
        u(0,x) &= \psi(x),  &\; &x\in \R^d.
    \end{aligned}
    \end{equation*}
    Hence, we are in the realm of the claim.
    Further, since $\hat{X}$ is a time-homogeneous Markov process, we have
    for all $(s,x)\in (0,T]\times \R^d$,
    \begin{equation}\label{eq:u_forward}
    \begin{aligned}
        u(s,x) = v(T-s,x) &= \E{\left. \psi(\hat{X}_T)\exp\left(-\int_{T-s}^T
            k(\hat{X}_\tau) \id\tau\right) \right| \hat{X}_{T-s} = x} \\
        &= \E{\left. \psi(\hat{X}_s)\exp\left(-\int_0^s k(\hat{X}_\tau)
            \id\tau\right) \right| \hat{X}_0 = x}.
    \end{aligned}
    \end{equation}
    Therefore, replacing $s$ by $t$ in the above equation \eqref{eq:u_forward} 
    proves the assertion.
\end{proof}

In view of our original problem, the Fokker-Planck equation \eqref{eq:num_IVP} 
that we want to solve numerically, in this case, reads 
for all $n\in\{1,\dots,N\}$ as
\begin{equation*}
    \begin{aligned}
    \frac{\partial q^n}{\partial t} (t,z) &= \hat{A}q^n(t,z) + r(z)q^n(t,z), 
    &\; & (t,z)\in (t_{n-1},t_n]\times\R^d,\\
        q^n(0,z) &= {p}^{n-1}(z), &\;    & z\in \R^d.
    \end{aligned}
\end{equation*}
Therefore, considering $k = -r$,
and assuming that $-r$ is non-negative
in \eqref{eqn:initialValuePDE}, Corollary~\ref{thm:initialValuePDE} gives,
for all $n\in\{1,\dots,N\}$, $t\in (t_{n-1},t_n]$, $z\in \R^d$, 
the representation
\begin{equation*}
    q^n(t,z) = \E{\left. p^{n-1}(\hat{X}_t)
        \exp\left(\int_{t_{n-1}}^t r(\hat{X}_\tau) \id\tau\right) \right|
        \hat{X}_{t_{n-1}} = z}.
\end{equation*}

\smallskip

To be explicit, in the next subsection we formulate two specific examples of filtering problems and 
show how they fit into the framework developed thus far by providing the 
auxiliary diffusion and conditional expectation representations for each of these cases.

\subsection{Two simple examples of filtering models}

The following are two simple examples for filtering problems, which will be used as benchmarks in the numerical studies. The results from the previous section hold true for these examples, even though the corresponding coefficients do not satisfy the uniform boundedness.
The linear filter in Example~\ref{ex:linear_filter} below is formulated in arbitrary finite dimensions.
Additionally, we give in Example~\ref{ex:benes-filter} the model for the purely one-dimensional, 
but nonlinear, Benes filter. For more details on the presented examples the reader may consult \cite[Chapter~6]{bain_crisan_2008}

\begin{example}[Linear Filter]\label{ex:linear_filter}
    For the Kalman filter we have the signal process given by the coefficient
    functions
    \begin{equation*}
        f(x) = M x + \eta \; \text{ and }\;
        \sigma(x) = \Sigma
    \end{equation*}
    and the observation process is determined by the sensor function
    \begin{equation*}
        h(x) = H  x + \gamma.
    \end{equation*}
    In this case, when $X_0$ is assumed normally distributed, the solution 
    $\pi_t$ of the filtering problem is known to be a Gaussian distribution
    with known mean and covariance, 
    see for example~\cite[Chapter 6.2]{bain_crisan_2008}.
    Then,
    for the linear filter, we see that the auxiliary diffusion process
    takes the form
    \begin{equation*}
        \hat{X}_t = \hat{X}_0 - \int_0^t M \hat{X}_s +\eta\, \der s +
            \int_0^t \Sigma \,\der \hat{W}_s,
    \end{equation*}
    and is thus the well-known Ornstein-Uhlenbeck process, plus an additional drift represented by $\eta$, with explicit 
    representation, in terms of the usual matrix exponential,
    \begin{equation*}
        \hat{X}_t = \exp\{-Mt\}
        \left( \hat{X}_0 + \int_0^t \exp\{Ms\}\Sigma \,\der \hat{W}_s\right).
    \end{equation*}
    Moreover, $r(x) = -\diver f(x) = - \operatorname{Tr} M$. Then the method for the 
    linear filter is given by the representation
    \begin{equation*}
        q^n(t,z) = \E{\left. p^{n-1}(\hat{X}_t)
            \exp\left(- \operatorname{Tr} M (t - t_{n-1})\right) \right|
            \hat{X}_{t_{n-1}} = z}.
    \end{equation*}
\end{example}

\smallskip
\begin{example}[Benes Filter]\label{ex:benes-filter}
    For the Benes filter we have one-dimensional signal and observation processes.
    The signal is given by the coefficient functions
    \begin{equation*}
        f(x) = \alpha\sigma \tanh (\beta+\alpha x/ \sigma) \;\text{ and } \;
        \sigma(x) \equiv \sigma \in \R,
    \end{equation*}
    where $\alpha, \beta \in \R$
    and the observation is given by the affine-linear sensor function
    \begin{equation*}
        h(x) = h_1 x + h_2,
    \end{equation*}
    with $h_1, h_2 \in \R$. Note that here we have given a special case of the 
    more general class of Benes filters, see~\cite[Chapter 6.1]{bain_crisan_2008}.
    Now, similar to the previous example, we compute the auxiliary diffusion
    \begin{equation*}
        \hat{X}_t = \hat{X}_0 - \int_0^t \alpha\sigma\tanh(\beta + \alpha x / \sigma) \, \der s +
            \int_0^t \sigma \,\der \hat{W}_s,
    \end{equation*}
    and the coefficient
    \begin{equation*}
        r(x) = - \diver f(x) = 
            -\alpha^2 \operatorname{sech}^2(\beta + \alpha x / \sigma).
    \end{equation*}
    This yields the scheme for the Benes case to be derived from the representation
    \begin{equation*}
        q^n(t,z) = \E{\left. p^{n-1}(\hat{X}_t)
            \exp\left(- \int_{t_{n-1}}^t \alpha^2\operatorname{sech}^2(\beta + \alpha \hat{X}_\tau /\sigma) \,\der \tau\right) \right|
            \hat{X}_{t_{n-1}} = z}.
    \end{equation*}
\end{example}

In the following subsection we introduce the optimisation problem associated with the filtering problem discussed above, and which is based on the simulation of the auxiliary diffusion. 

\subsection{Optimisation problem for the prior}\label{sec:optim}
Above we have found a Feynman-Kac representation
for the solution of the Fokker-Planck equation in the case of
smooth coefficients of the signal process.
In analogy to \cite[Proposition 2.7]{2018arXiv180600421B}
we formulate the following result about a minimisation property
 in Proposition~\ref{thm:minimizer} below.

\begin{prop}\label{thm:minimizer}
    Let $d\in\N$, $T>0$, $a<b\in\R$, $k\in C(\R^d, [0,\infty))$,
    let $\hat{A}$ be the operator defined in
    Definition~\ref{def:A_hat},
    and let $\psi:\R^d \rightarrow \R$ be an at most polynomially growing
    function.
    Suppose that $u \in C_b^{1,2}((0,T]\times \R^d,\R)$ 
    satisfies the Cauchy problem
    \begin{equation*}\label{eqn:hmm}
    \begin{aligned}
         \frac{\partial u}{\partial t}(t,x) + k(x)u(t,x) &= \hat{A}u(t,x),  
         &\; & (t,x)\in (0,T] \times \R^d,\\
        u(0,x) &= \psi(x),  &\; &x\in \R^d,
    \end{aligned}
    \end{equation*}
let $\xi: \Omega \rightarrow  [a,b]^d$ be a continuous, uniformly distributed
$\mathcal{F}_0$-random variable and let $\hat{\mathbb{X}}$ be the diffusion
generated by $\hat{A}$ and with the property that, \Pas{}, 
$\hat{\mathbb{X}}_0 = \xi$.
Then there exists a unique continuous function $U:[a,b]^d \rightarrow \R$
such that
\begin{multline*}
    \E{\left| \psi(\hat{\mathbb{X}}_T)\exp\left(-\int_0^T
        k(\hat{\mathbb{X}}_\tau) \id\tau\right) - U(\xi)\right|^2}\\ =
        \inf_{v \in C([a,b]^d,\R)} \E{\left| \psi(\hat{\mathbb{X}}_T)
        \exp\left(-\int_0^T k(\hat{\mathbb{X}}_\tau) \id\tau\right) -
        v(\xi)\right|^2}
\end{multline*}
and for all $x \in [a,b]^d$ we have $U(x)=u(T,x)$.
\end{prop}
\begin{proof}
    Let $T>0$. For all $x\in\R^d$, let $\hat{X}^x$ be the $\hat{A}$-diffusion 
    starting at $x$.
    Since, by assumption, $k$ is non-negative and $\psi$ has polynomial growth
    it follows that there exist real numbers $L>0$ and $\lambda\geq1$ such
    that for all $x\in\R^d$,
    \begin{equation}\label{eqn:firstL2}
        \E{\left| \psi(\hat{X}^x_T) \exp\left(-\int_0^T k(\hat{X}^x_\tau)
        \id\tau\right) \right|^2}  \leq \E{\left|
        L(1+ \| \hat{X}^x_T \|^{2\lambda}) \right|^2} < \infty.
    \end{equation}
    Further, because the map 
    \begin{equation*}
        \R^d \ni x \mapsto  \psi(\hat{X}^x_T) \exp\left(-\int_0^T
            k(\hat{X}^x_\tau) \id\tau\right)
    \end{equation*}
    is continuous and at most polynomially growing,
     \cite[Lemma 2.6]{2018arXiv180600421B} implies that the function
    \begin{equation}\label{eqn:secondcont}
        \R^d \ni x \mapsto \E{ \psi(\hat{X}^x_T) \exp\left(-\int_0^T k(\hat{X}^x_\tau)
            \id\tau\right) }
    \end{equation}
    is continuous. 
    Note that
    the function
    \begin{equation}\label{eqn:thirdmeas}
        \R^d \times \Omega \ni (x,\omega ) \mapsto \psi( \hat{X}^x_T(\omega ) )
            \exp\left(-\int_0^T k(\hat{X}^x_\tau(\omega))\id\tau \right)
    \end{equation}
    is $\mathcal{B}([a,b]^d) \otimes \mathcal{F}/ \mathcal{B}(\R^d)$-measurable.
    Finally, by virtue of \eqref{eqn:firstL2}, \eqref{eqn:secondcont},
    \eqref{eqn:thirdmeas}, and \cite[Proposition 2.2]{2018arXiv180600421B},
    there exists a unique continuous function $U:[a,b]^d \rightarrow \R$ 
    such that
    \begin{align*}
        &\int_{[a,b]^d}  \E{\left| \psi(\hat{X}^x_T)\exp\left(-\int_0^T
            k(\hat{X}^x_\tau) \id\tau\right) - U(x)\right|^2} \id x\\
        &= \inf_{v \in C([a,b]^d,\R)}\int_{[a,b]^d} \E{\left| \psi(\hat{X}^x_T)
            \exp\left(-\int_0^T k(\hat{X}^x_\tau) \id\tau\right) -
            v(x)\right|^2} \id x
    \end{align*}
    and such that for all $x \in [a,b]^d$ we have
    \begin{equation*}
        U(x)=\E{\psi(\hat{X}^x_T)\exp\left(-\int_0^T k(\hat{X}^x_\tau)
        \id\tau\right)}.
    \end{equation*}
    Now, for all $V\in C([a,b]^d,\R)$ we have that the map
    \begin{equation*}
        C([0,T],\R^d) \ni \gamma \mapsto 
        \left|\psi(\gamma_T)\exp\left(-\int_0^T k(\gamma_\tau)\id\tau\right)-
        V(\gamma_0)\right|^2 \in \R     
    \end{equation*}
    is at most polynomially growing. Thus \cite[Lemma 2.6]{2018arXiv180600421B}
    implies that for all $V\in C([a,b]^d,\R)$ we have that
    \begin{equation*}
        \begin{aligned}
        &\E{ \left| \psi(\mathbb{X}_T)\exp\left(-\int_0^T k(\mathbb{X}_\tau)\id
         \tau\right) - V(\xi) \right|^2}\\
        &=
        \frac{1}{(b-a)^d} \int_{[a,b]^d} 
        \E{ \left| \psi(X_T^x)\exp\left(-\int_0^T k(X_\tau^x)\id
        \tau\right) - V(x) \right|^2} \id x.
        \end{aligned}
    \end{equation*}
    Then, for all $V\in C([a,b]^d,\R)$ with the property that
    \begin{equation*}
        \begin{aligned}
        &\E{ \left| \psi(\mathbb{X}_T)\exp\left(-\int_0^T k(\mathbb{X}_\tau)\id
         \tau\right) - V(\xi) \right|^2}\\
        &=
        \inf_{v\in C([a,b]^d,\R)}
        \E{ \left| \psi(\mathbb{X}_T)\exp\left(-\int_0^T k(\mathbb{X}_\tau)\id
        \tau\right) - v(\xi) \right|^2},
        \end{aligned}
    \end{equation*}
    a direct calculation shows that
    \begin{equation*}
        \begin{aligned}
        &\int_{[a,b]^d} \E{ \left| \psi({X}_T^x)
        \exp\left(-\int_0^T k({X}^x_\tau)\id
         \tau\right) - V(x) \right|^2} \id x\\
        &=
        \inf_{v\in C([a,b]^d,\R)}
        \int_{[a,b]^d}
        \E{ \left| \psi(X_T^x)\exp\left(-\int_0^T k(X_\tau^x)\id
        \tau\right) - v(\xi) \right|^2} \id x.
        \end{aligned}
    \end{equation*}
    Hence also this minimiser is unique and equals $U$ and finally
    \begin{equation*}
        \begin{aligned}
        &\E{ \left| \psi(\mathbb{X}_T)\exp\left(-\int_0^T k(\mathbb{X}_\tau)\id
         \tau\right) - U(\xi) \right|^2}\\
        &=
        \inf_{v\in C([a,b]^d,\R)}
        \E{ \left| \psi(\mathbb{X}_T)\exp\left(-\int_0^T k(\mathbb{X}_\tau)\id
        \tau\right) - v(\xi) \right|^2}.
        \end{aligned}
    \end{equation*}
    This together with the Feynman-Kac formula proves the result.
\end{proof}

Proposition \ref{thm:minimizer} guarantees that we have a feasible minimisation
problem to approximate by the learning algorithm.

In the following section we will describe the machine learning algorithm used to approximate the PDE using the above optimisation representation. 
Furthermore, we derive the Monte-Carlo method used to approximate the normalisation constant in the correction step.
We thus specify our full method.

\section{Splitting method for the neural network representation of the posterior}\label{sec:DL_method}
Here, we introduce some of the terminology specific to the field of neural networks.
For an in-depth discussion on deep learning terminology, algorithms and applications, we refer the reader to the book \cite{goodfellow2016deep}. Thereafter, we specify explicitly the learning algorithm employed in our method. Subsequently, we derive the Monte-Carlo method used in the correction step of the splitting method and the section ends with a full description of our algorithm in pseudocode.

\subsection{Neural network model for prediction step}\label{sec:NN_model}

\begin{defi}
    Given $L \in \mathbb{N}$ and $(l_0,\dots,l_L) \in \mathbb{N}^{L+1}$ and a
    continuous function $\rho \in C^0(\R)$ a (feed-forward fully-connected) neural network $\mathcal{NN}$
    is a function $\R^{l_0} \rightarrow \R^{l_L}$ given by
    \begin{equation*}
        \mathcal{NN}(x) = \left( \bigcirc_{i=1}^L \rho \odot
            \mathcal{A}_i^{(l_{i-1}, l_i)} \right) (x),
    \end{equation*}
    where the $\mathcal{A}_i^{(l_i-1, l_i)}$ are affine maps
    $\R^{l_{i-1}} \rightarrow \R^{l_i}$ of the form
    $x \mapsto A_i x + b_i$, $A_i\in\R^{l_{i-1}\times l_i}$, $b_i \in \R^{l_i}$.
\end{defi}

The number $L$ is called the \emph{depth} of the network,
the function $\rho$ is called the \emph{activation function},
and the matrices and vectors $A_i$ and $b_i$ are called the \emph{weights}
and \emph{biases} of the $i$-th \emph{hidden layer}, respectively. In the experimental part of this work, we consider
the activation function $\rho(x)=\tanh(x)$. Other common choices include the ReLu activation function $\operatorname{ReLu}(x) = \max\{0,x\}$ or the sigmoidal function $\sigma(x)=1/(1+\exp(-x))$, among many others.
Collectively, the parameters of the function represented by the neural network are denoted
\begin{equation*}
\theta = \{ \{A_i^{jk}\}_{jk}, \{b_i^{j}\}_j \colon i=1,\dots, L\} \subset \R^{(\sum_{i=2}^L l_{i-1}l_i+l_i)}
\end{equation*}
and we sometimes write $\mathcal{NN}_\theta$ to note the dependence explicitly.
The symbol $\bigcirc$ denotes function composition and the symbol
$\odot$ denotes componentwise function composition, i.e. for any $\mathcal{A}\colon \R^{m} \rightarrow \R^{n}$ and $x\in\R^m$ we have
\begin{equation*}
    (\rho \odot \mathcal{A})(x) = (\rho((Ax)_1), \dots, \rho((Ax)_n))^\prime \in \R^n.
\end{equation*}

In general, the weights and biases of a neural network are to be chosen freely
and are commonly determined using an optimisation algorithm such as gradient descent, 
stochastic gradient descent~\cite{bottou-mlss-2004} or variants thereof,
such as AdaGrad~\cite{Adagrad_paper}, momentum methods~\cite{rumelhart86}, or the ADAM optimiser~\cite{kingma2014adam}.
Our method of choice in this work is the ADAM optimiser.
The optimisation procedure based on supplied \emph{training data} is in this context commonly referred
to as \emph{learning}. Notice, however, that there is an important distinction between learning and optimisation. While optimisation is concerned with the pure minimisation (or maximisation) of a target function, the goal of learning is to create a model that \emph{generalises} well, i.e. performs well on unseen inputs. Thus, in certain contexts it is undesirable to fit a model too closely to the provided training data, since this can degrade the out-of-sample performance, a phenomenon known as \emph{overfitting}.

Moreover, the initialisation of the parameters is a
crucial part of the performance of the optimisation and defines its own branch of research within machine learning.
Additionally, the neural network model has various free parameters that are
neither given by the original problem nor are they determined by the learning procedure.
These include the architecture of the network, i.e. the depth $L$, the layer widths $l_i$, or certain parameters in the optimisation algorithm such as the \emph{learning rate} (i.e. the step size
of the gradient descent method) or training batch size, and are commonly to be chosen heuristically or from experience. These are commonly called \emph{hyperparameters}.

Additionally, we employ the technique of batch-normalisation~\cite{ioffe2015batch} in our computations, but refrain here from a detailed discussion. The reader is referred to the original work~\cite{ioffe2015batch} or the book~\cite{goodfellow2016deep}.

In order to use a neural network model for the filtering problem, we employ the splitting-up
method to first split the problem into the solution of a deterministic Fokker-Planck PDE and the 
subsequent inclusion of the observation using the likelihood and normalisation procedure.

The PDE step is where we incorporate the deep learning method to solve the Fokker-Planck equation over a rectangular domain $\Omega_d = [\alpha_1, \beta_1]\times\dots\times[\alpha_d,\beta_d]$, for the sake of computational feasibility.
Its solution is reformulated into the optimisation problem over function space given in Proposition~\ref{thm:minimizer}. This optimisation problem is approximated by the optimisation
\begin{equation*}
    \inf_{\theta \in \R^{\sum_{i=2}^L l_{i-1}l_i+l_i}} \E{\left| \psi(\hat{\mathbb{X}}_T)
        \exp\left(-\int_0^T k(\hat{\mathbb{X}}_\tau) \id\tau\right) -
        \mathcal{NN}_\theta(\xi)\right|^2}
\end{equation*}
where the solution of the PDE is represented by a neural network and the infinite-dimensional function space has been parametrised by the neural network parameters $\theta$.
To this problem we will be able to apply a gradient descent method for the determination of the
parameters in the model to minimise the associated \emph{loss function}
\begin{equation*}
\mathcal{L}(\theta ; \{\xi^i, \{\hat{\mathbb{X}}_{\tau_j}^i\}_{j=0}^J \}_{i=1}^{N_b}) =
    \frac{1}{N_b} \sum_{i=1}^{N_b} \left| 
        \psi(\hat{\mathbb{X}}_T^i)
        \exp(- \sum_{j=0}^{J-1} k(\hat{\mathbb{X}}_{\tau_j}^i) (\tau_{j+1}-\tau_j)) -
        \mathcal{NN}_\theta(\xi^i)\right|^2,
\end{equation*}
where $N_b$ is the batch size and $ \{\xi^i, \{\hat{\mathbb{X}}_{\tau_j}^i\}_{j=0}^J \}_{i=1}^{N_b}$ is a training batch of independent identically distributed realisations $\xi^i$ of $\xi \sim \mathcal{U}(\Omega_d)$ and $\{\hat{\mathbb{X}}_{\tau_j}^i\}_{j=0}^J$ the approximate i.i.d.{} realisations of sample paths of the auxiliary diffusion started at $\xi^i$ over the time-grid $\tau_0=0<\tau_1<\cdots<\tau_{J-1}<\tau_J=T$. The sample paths are, for example, approximated using the Euler-Maruyama or a similiar SDE simulation method~\cite{KloedenPlaten1992}.
In practice, since the solution of the Fokker-Planck equation we seek is non-negative, we usually augment the loss $\mathcal{L}$ by an additional term to encourage positivity of the neural network and use
\begin{equation*}
    \tilde{\mathcal{L}}(\theta ;  \{\xi^i, \{\hat{\mathbb{X}}_{\tau_j}^i\}_{j=0}^J \}_{i=1}^{N_b}) =
    {\mathcal{L}}(\theta ;  \{\xi^i, \{\hat{\mathbb{X}}_{\tau_j}^i\}_{j=0}^J \}_{i=1}^{N_b}) 
    + \lambda \sum_{i=1}^{N_b} \max\{0,\mathcal{NN}_\theta(\xi^i)\}
\end{equation*}
with the hyperparameter $\lambda$ to be chosen.

Thus, in the notation of subsection~\ref{sec:F_eq}
we replace the Fokker-Planck solution by a neural network model,
i.e. we \emph{postulate} a neural network model
\begin{equation*}
    \tilde{p}_n(z) = \mathcal{NN}(z),
\end{equation*}
with support on $\Omega_d$. Therefore we require the a priori chosen domain to capture most of the mass of the probability distribution it is approximating.

\subsection{Monte-Carlo correction step}\label{sec:mc_correction}
We then realise the correction step via Monte-Carlo sampling over the bounded rectangular domain $\Omega_d$ to approximate the integral
\begin{equation*}
     \int_{\R^d} \xi_n(z)\mathcal{NN}(z) \id z = \int_{\R^d} \exp \left( -\frac{t_n-t_{n-1}}{2} 
        ||z_n - h(z)||^2 \right) \mathcal{NN}(z) \id z,
\end{equation*}
where, as defined earlier, $z_n = \frac{1}{t_n-t_{n-1}}(Y_{t_n}-Y_{t_{n-1}})$. Now, since the neural network has $\operatorname{supp}(\mathcal{NN}) \subseteq \Omega_d$ this is equal to the integral
\begin{equation}\label{eq:MC_int}
    \int_{\Omega_d} \exp \left( -\frac{t_n-t_{n-1}}{2} 
        ||z_n - h(z)||^2 \right) \mathcal{NN}(z) \id z.
\end{equation}
In general, to achieve the approximation of the above integral via Monte-Carlo, one needs to be able to sample from an appropriate density. Moreover, see Remark~\ref{rem:mc_sampling} below for possible alternatives.
   
\begin{remark}\label{rem:mc_sampling}
The usage of the Monte-Carlo method to perform the normalisation 
is optional in our low-dimensional experimental setup below, where efficient quadrature methods are a 
good alternative.
However, we chose to design our algorithm around the sampling based method, as a
large part of the literature devoted to machine learning algorithms for PDEs aims to 
design grid-free (in space) methods to achieve better performance in high dimensions.
In that regard, we specify our algorithm so that it can be tested in higher dimensional, grid-free, settings without major alterations in subsequent studies.
\end{remark}

Since, in this work, we are considering mainly affine-linear sensor functions $h(x) = h_1 x + h_2$, we illustrate the Monte-Carlo integration method in this case.
Notice that the likelihood function then reads
\begin{align*}
    \xi_n(z) &= \exp \left( -\frac{t_n-t_{n-1}}{2} 
        (z_n - h_1 z - h_2)^2 \right) \\
    &= \exp \left( -\frac{(t_n-t_{n-1})h_1^2}{2} 
        \left(\frac{z_n - h_2}{h_1} - z \right)^2 \right)\\
    &= \exp \left( -\frac{1}{2} 
        \left( \frac{ \frac{z_n - h_2}{h_1} - z }{((t_n-t_{n-1})h_1^2)^{-1/2}} \right)^2 \right)\\
    &= \frac{\sqrt{2\pi}}{\sqrt{(t_n-t_{n-1})h_1^2}} \mathcal{N}_{\text{pdf}}\left(\frac{z_n - h_2}{h_1}, \frac{1}{(t_n-t_{n-1})h_1^2}\right)(z),
\end{align*}
where $\mathcal{N}_{\text{pdf}}(\mu,\sigma^2)$ denotes the probability density function of a normal distribution with mean $\mu$ and variance $\sigma^2$. Therefore, we can write the integral \eqref{eq:MC_int} as
\begin{equation*}
    \frac{\sqrt{2\pi}}{\sqrt{(t_n-t_{n-1})h_1^2}} \mathbb{E}_{Z}[\mathcal{NN}(Z)]; \qquad  \qquad Z \sim \mathcal{N}\left(\frac{z_n - h_2}{h_1}, \frac{1}{(t_n-t_{n-1})h_1^2}\right).
\end{equation*}
As it is straightforward to numerically sample from a Gaussian distribution, the Monte-Carlo
approximation derived above is implementable so that we can compute the normalisation constant $C_n$ numerically. Thus, we can explicitly represent the approximate posterior density
\begin{equation*}
        p^n(z) = \frac{1}{C_n} \xi_n(z) \tilde{p}^n(z),
\end{equation*}
and use it as the initial condition for the next time iteration. Therefore, our scheme is fully recursive and can be applied sequentially.

\begin{remark}
    Additional techniques to adjust the support of the approximation are needed when iterating
    the scheme over a long time duration/many steps as, eventually, in many common filtering setups, it will be the case that the mass of the posterior moves outside the initial domain. 
    The way to mitigate this problem depends, in general, on the specific filtering model under consideration and will be subject of further investigation.
\end{remark}

\subsection{Algorithm summary}\label{sec:algorithm}
We briefly summarise our full approximation method. In Algorithm~\ref{alg:splitting} we present the pseudocode for the splitting method as we apply it to the filtering equation. The method is designed to be fully grid-free in space, for the reasons outlined above in Remark~\ref{rem:mc_sampling}. 
Furthermore, a main feature of our algorithm is the ability to iterate it over successive time steps so that observations may arrive sequentially, and there is no strict requirement for them to be available beforehand. 
This is an especially important property in real-world filtering scenarios where observations are typically processed \emph{online}. Therefore, Algorithm~\ref{alg:splitting} is formulated as an iterative procedure over the observation time-grid $0=t_0, t_1,\dots, t_N=T$.

\begin{algorithm}[t]
    \caption{Splitting method for neural network representation of posterior}
    \label{alg:splitting}
    \begin{algorithmic}[1]
        \Require Time-grid $0=t_0, t_1,\dots, t_N=T$
        \Require Initial density $p_0$ 
        \Require Observations $Y_0,\dots, Y_N$
        \Require Affine-linear sensor function $h(x)=h_1 x+ h_2$
        \Function{Posterior$_0$}{$x$}
            \State\Return $p_0(x)$
        \EndFunction
        \For{$n$ from $1$ to $N$}
            \State Initialize $\mathcal{NN}^n_{\theta_{init}}$
            \State $\mathcal{NN}^n_{\theta_{trained}} \gets$\textsc{TrainNet}($\mathcal{NN}^n_{\theta_{init}}$, \textsc{Posterior}$_{n-1}$)
            \State Compute $z_n = \frac{1}{t_n-t_{n-1}}(Y_{t_n}-Y_{t_{n-1}})$
            \State Draw $N_{samples}$ samples $Z_j$ from $\mathcal{N}\left(\frac{z_n - h_2}{h_1},\frac{1}{(t_n-t_{n-1})h_1^2}\right)$
            \State Compute $C_n = \frac{1}{N_{samples}}\sum_{j=1}^{N_{samples}} \mathcal{NN}^n_{\theta_{trained}}(Z_j)$
            \Function{Posterior$_n$}{$x$}
                \State\Return $\frac{1}{C_n} \exp \left( -\frac{t_n-t_{n-1}}{2} 
        (z_n - h(x))^2 \right) \mathcal{NN}^n_{\theta_{trained}}(x)$
            \EndFunction
        \EndFor
    \end{algorithmic}
\end{algorithm}

Algorithm~\ref{alg:splitting} includes a network training step which we clarify in the pseudocode presented in Algorithm~\ref{alg:training}. Note that we give here, in the interest of clarity, a simplified version of the actual gradient-descent method that we employ in the numerical studies in section~\ref{sec:num_res}. However, the general rationale behind both methods is the same gradient-descent based process. The important parameters of the learning method are the number of training steps, usually called \emph{epochs}, $N_{epochs}$, the training batch size $N_b$ as well as the learning rate $\kappa$ that determines the step size of the gradient descent step to adjust the parameters of the neural network. In our studies below, we chose an adaptive learning rate based on a \emph{learning rate schedule}. That is, we choose a set of integers $0=K_0 < K_1 < \dots < K_M < K_{M+1}=\infty$ as cut-off steps, and a set of learning rates $\kappa_0,\dots,\kappa_M$ and adjust the learning rate during the training procedure according to
\begin{equation*}
    \kappa(n) = \sum_{i=0}^{M} \kappa_i \mathbf{1}_{[K_i, K_{i+1})}(n), \qquad n=1,\dots, N_{epochs}.
\end{equation*}
Since the training method is based on $N_b$ samples of the auxiliary diffusion in each epoch,
the full training uses $N_b N_{epochs}$ independent Monte-Carlo samples in total.

\begin{algorithm}[b]
    \caption{Network training (simplified)}
    \label{alg:training}
    \begin{algorithmic}[1]
    \Require $N_{epochs}$, $N_{b}$, $\kappa$
        \Function{TrainNet}{$\mathcal{NN}_{\theta_{init}}$, \textsc{Posterior}$_0$}
            \State $\theta\gets \theta_{init}$
            \For{$n$ from $1$ to $N_{epochs}$}
                \State Draw $N_{b}$ samples $\{\xi^i, \{\hat{\mathbb{X}}_{\tau_j}^i\}_{j=0}^J\}$
                \State Compute $\nabla_{\theta}\tilde{\mathcal{L}}(\theta ;  \{\xi^i, \{\hat{\mathbb{X}}_{\tau_j}^i\}_{j=0}^J \}_{i=1}^{N_b}) $
                \State $\theta \gets \theta - \kappa \nabla_{\theta}\tilde{\mathcal{L}}(\theta ;  \{\xi^i, \{\hat{\mathbb{X}}_{\tau_j}^i\}_{j=0}^J \}_{i=1}^{N_b})$ \Comment{Gradient descent}
            \EndFor
            \State $\theta_{trained}\gets \theta$
            \State\Return $\mathcal{NN}_{\theta_{trained}}$
        \EndFunction
    \end{algorithmic}
\end{algorithm}

Figure~\ref{fig:architecture} illustrates the neural network architecture that we are using in the numerical experiments exhibited in section~\ref{sec:num_res}. This architecture is inspired by the one used in previous experiments by other authors, for example in \cite{Han8505}.
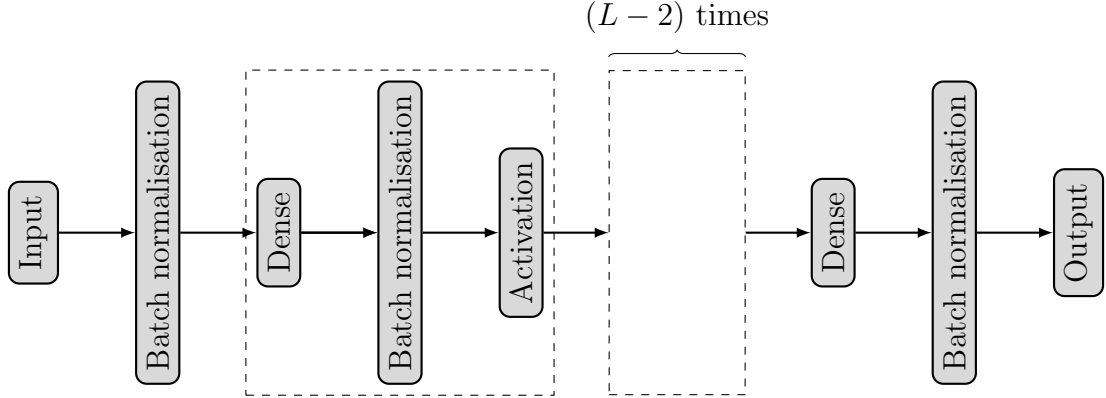
\begin{figure}
	\centering
	\begin{tikzpicture}[remember picture]
		\node[
			rotate=90,
			minimum width=1cm,
			inner xsep=5pt,
			fill=black!15,
			draw=black,
			thick,
			rounded corners,
			] (input) at (0,0)
			{
				Input
			};
			
		\node[
			right=of input.south,
			rotate=90,
			anchor=north,
			minimum width=1cm,
			inner xsep=5pt,
			fill=black!15,
			draw=black,
			thick,
			rounded corners,
			] (batchnorm1) at (input.south)
			{
				Batch normalisation
			};
			
		\node[
			right=of batchnorm1.south,
			rotate=90,
			anchor=north,
			minimum width=1cm,
			inner xsep=5pt,
			fill=black!15,
			draw=black,
			thick,
			rounded corners,
			] (dense1) at (batchnorm1.south)
			{
				Dense
			};
		\node[
			right=of dense1.south,
			rotate=90,
			anchor=north,
			minimum width=1cm,
			inner xsep=5pt,
			fill=black!15,
			draw=black,
			thick,
			rounded corners,
			] (batchnorm2) at (dense1.south)
			{
				Batch normalisation
			};
		\node[
			right=of batchnorm2.south,
			rotate=90,
			anchor=north,
			minimum width=1cm,
			inner xsep=5pt,
			fill=black!15,
			draw=black,
			thick,
			rounded corners,
			] (activation) at (batchnorm2.south)
			{
				Activation
			};		
		\node[draw=black, dashed, fit= (dense1) (batchnorm2) (activation.south)] (block) {};
		
		\node[
			right=of activation.south,
			rotate=90,
			anchor=north,
			minimum width=1cm,
			minimum height=1.5cm,
			inner xsep=5pt,
			text=cyan,
			] (blockl2times-dummy) at (activation.south)
			{
				Batch normalisation
			};
		\draw[fill=white,draw=white] (blockl2times-dummy.north west) rectangle (blockl2times-dummy.south east);
		\node[draw=black, dashed, fit= (blockl2times-dummy)] (blockl2times) {};
		\draw[
			decorate,
			decoration={brace},
			] ($(blockl2times.north west) + (0,5pt)$) --node[above=5pt]{$(L-2)$ times}  ($(blockl2times.north east) + (0,5pt)$);
		
		\node[
			right=of blockl2times-dummy.south,
			rotate=90,
			anchor=north,
			minimum width=1cm,
			inner xsep=5pt,
			fill=black!15,
			draw=black,
			thick,
			rounded corners,
			] (dense2) at (blockl2times-dummy.south)
			{
				Dense
			};
			
		\node[
			right=of dense2.south,
			rotate=90,
			anchor=north,
			minimum width=1cm,
			inner xsep=5pt,
			fill=black!15,
			draw=black,
			thick,
			rounded corners,
			] (batchnorm3) at (dense2.south)
			{
				Batch normalisation
			};
			
		\node[
			right=of batchnorm3.south,
			rotate=90,
			anchor=north,
			minimum width=1cm,
			inner xsep=5pt,
			fill=black!15,
			draw=black,
			thick,
			rounded corners,
			] (output) at (batchnorm3.south)
			{
				Output
			};
			
		\draw[thick, -latex] (input.south) -- (batchnorm1.north);
		\draw[thick, -latex] (batchnorm1.south) -- (dense1.north);
		\draw[thick, -latex] (dense1.south) -- (batchnorm2.north);
		\draw[thick, -latex] (batchnorm2.south) -- (activation.north);
		\draw[thick, -latex] (dense1.south) -- (batchnorm2.north);
		\draw[thick, -latex] (activation.south) -- (blockl2times.west);
		\draw[thick, -latex] (blockl2times.east) -- (dense2.north);
		\draw[thick, -latex] (dense2.south) -- (batchnorm3.north);
		\draw[thick, -latex] (batchnorm3.south) -- (output.north);
	\end{tikzpicture}
	\caption{Neural network architecture used in our experiments. We use the architecture similar to the one employed in \cite{Han8505}. The input is initially transformed by a batch-normalisation layer \cite{ioffe2015batch} and then a sequence of a triple (dashed box) consisting of an affine linear (Dense) transformation, a batch normalisation, and a subsequent application of the $tanh$-nonlinearity (Activation) is applied $L-1$ times, where $L$
	is the depth of the neural network. Before returning, another affine transformation (Dense) and then a final batch-normalisation are applied.}
	\label{fig:architecture}
\end{figure}

\section{Numerical results for the splitting scheme}\label{sec:num_res}
We implement Algorithm~\ref{alg:splitting} for Examples~\ref{ex:linear_filter} and~\ref{ex:benes-filter} using
Tensorflow~\cite{tensorflow2015-whitepaper}. 
For a practical guide on the implementation of deep learning algorithms, the reader may consult~\cite{geron2019hands}.

In all examples below, the neural network architecture is a feed-forward fully connected neural network with a one-dimensional input layer, two hidden layers with a layer width of 51 neurons each, and an output layer of dimension one. For the optimisation algorithm we chose the ADAM optimiser and performed the training over 6002 epochs with a batch size of 600 samples. Note that during our testing we found that the batch size had a crucial effect on the performance of our algorithm. If chosen too small, the training procedure we used failed to discover an acceptable set of parameters for the neural network. If chosen too large, we observed that the training was slowed down on our hardware.

\subsection{One-dimensional linear filter}\label{sec:1d-lin_res}
Here we present the numerical results for the one-dimensional linear filtering setting outlined in Example~\ref{ex:linear_filter}. We first present in section~\ref{sec:OUlin_num} a filter that does not move outside the domain, based on an Ornstein-Uhlenbeck signal process.  Next, we show the results obtained for the linear filter with a signal process that moves toward the domain boundary in section~\ref{sec:lin_num}.

\subsubsection{Linear filter, case 1: \textit{M} = --1, \textit{\texteta} = 0}\label{sec:OUlin_num}
We are considering a linear filter with an Ornstein-Uhlenbeck signal process using the set of parameters, corresponding to the notation in Example~\ref{ex:linear_filter}, given in Table~\ref{tab:1d_oulin}. Moreover, as the initial condition we chose a Gaussian density with mean $0.0$ and standard deviation $0.01$. We iterate our method over $60$ timesteps up to a final time of $0.6$.

\begin{table}[h!]
\centering
\begin{tabular}{ c c c c c c c c }
$x_0$ & $y_0$ & $M$ & $\eta$ & $\Sigma$ & $H$ & $\gamma$ & $\Delta t$   \\
\hline\\[-1em]
$0.0$ & $0.0$ & $-1.0$ & $0.0$ & $0.1$ & $90.0$ & $0.0$ & $0.01$\\
\end{tabular}
\caption{Parameters used in the numerical experiment for the one-dimensional linear filter, case 1.}
\label{tab:1d_oulin}
\end{table}

The results of our approximation method applied to the linear filter with Ornstein-Uhlenbeck signal are visualised in Figure~\ref{fig:1d_oulin_graphs}. The full evolution of the estimated posterior is shown in Figure~\ref{fig:1d_oulin_graphs}~(a). In particular, we see that the approximated solution stays within the considered spatial domain $[-0.5, 0.5]$. This feature will be important when we discuss the linear filter with drift below. Moreover, note that in correspondence with the theoretical expectations, the variance of the approximated posterior distribution initially increases and then stays constant, with an oscillating mean which is affected by the sequentially arriving observations. In Figure~\ref{fig:1d_oulin_graphs}~(b)-(d) we present three snapshots of the numerical solution obtained with our modified splitting scheme. In each of the three graphs, the yellow line shows the plot of the neural network over the observed domain, approximating the solution of the Fokker-Planck PDE with initial condition given by the posterior density obtained from the previous step. The blue-shaded line is a pointwise Monte-Carlo reference solution based on the Sobol sequence over the spatial domain. This is used as a visual guide to judge the quality of the shape the neural network represents. Note that this is \emph{not} the theoretical solution for the filtering problem, but a reference solution for the Fokker-Planck equation for the prior, based on the initial condition given by the previous estimate.
The black dashed line shows the plot of the normalised posterior using the method outlined above. Additionally, we plotted the mean and standard deviation of the exact solution to the considered filtering problem as three blue vertical lines, the higher one representing the theoretical mean and the lower ones the standard deviation. The position of the signal is plotted as a red inverted triangle and the position of the observation as a green triangle. Note that the observation may lie outside the domain and thus may not be present in the graph.

\begin{figure}[!ht]
  \centering
  \parbox{\figrasterwd}{
    \parbox{.55\figrasterwd}{%
      \subcaptionbox{}{\includegraphics[width=1.2\hsize]{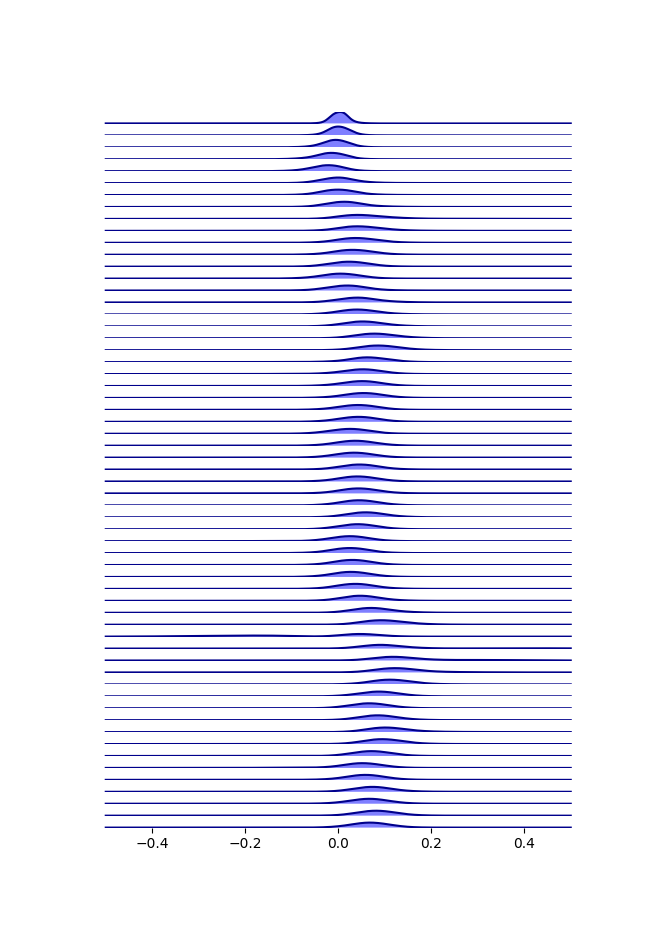}}
    }
    \hskip1em
    \parbox{.35\figrasterwd}{%
      \subcaptionbox{}{\includegraphics[width=\hsize]{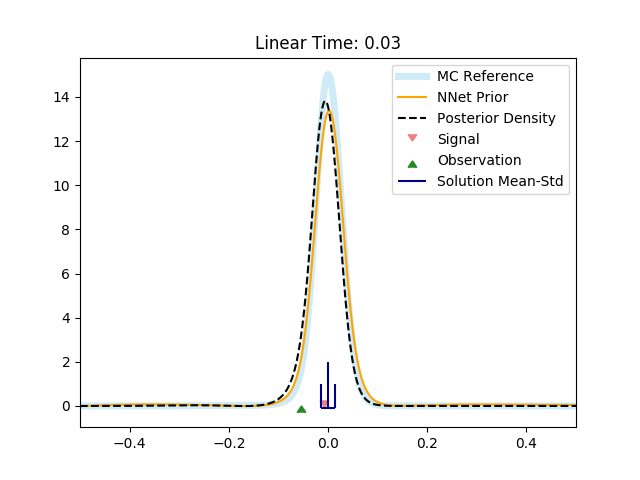}}
      \vskip1em
      \subcaptionbox{}{\includegraphics[width=\hsize]{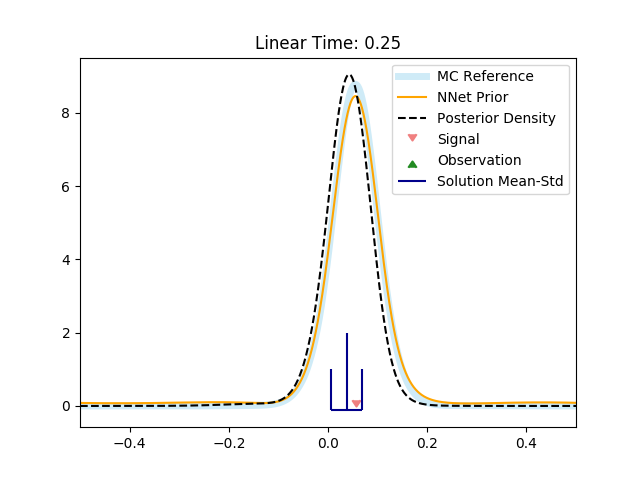}}
      \vskip1em
      \subcaptionbox{}{\includegraphics[width=\hsize]{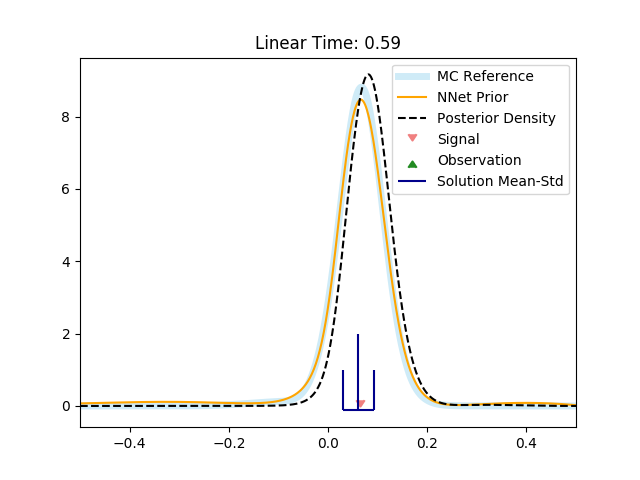}} 
    }
  }
  \caption{Results of the combined splitting-up/machine-learning approximation applied iteratively to the linear filtering problem, case 1. (a) The full evolution of the estimated posterior distribution produced by our method, plotted at all intermediate timesteps, from top to bottom. (b)-(d) Snapshots of the approximation at an early time, $t=0.03$, an intermediate time, $t=0.25$, and a late time, $t=0.59$, obtained after 3, 25 and 59 iterations of our method, respectively. The black dotted line in each graph shows the estimated posterior, the yellow line the prior estimate represented by the neural network, and the light-blue shaded line shows the Monte-Carlo reference solution for the Fokker-Planck equation.}
\label{fig:1d_oulin_graphs}
\end{figure}

The errors and diagnostics for the linear filter, case 1, are shown in Figure~\ref{fig:1d_oulin_graphserr}. Here, Figure~\ref{fig:1d_oulin_graphserr}~(a) is a graph of the absolute value of the difference between the mean of the approximate posterior and the theoretical posterior mean. We see that the error fluctuates about a constant value, which is the desired result. In particular, we do not expect a decreasing error but rather a stable one. This shows that the method is stable when iterated over many time steps. The two peaks at times $0.44$ and $0.46$ are explained below and due to a statistical outlier in the observation/likelihood.
Figure~\ref{fig:1d_oulin_graphserr}~(b) shows the training performance of the neural network approximation measured by the $L_2$-error over the domain between the Monte-Carlo reference solution of the Fokker-Planck PDE and the neural network representation across the training epochs. Each line in the graph represents a separate neural network, one for each timestep. Here we can see that the neural network training consistently converges to the Monte-Carlo solution across all time steps.
The probability mass of the neural net and Monte-Carlo reference solutions of the Fokker-Planck equation is plotted against time in Figure~\ref{fig:1d_oulin_graphserr}~(c), where we conclude that the machine learning approximation tends to slightly overestimate the mass of the solution.  
Lastly, in Figure~\ref{fig:1d_oulin_graphserr}~(d) we plot the acceptance rate of the Monte-Carlo integration of the neural network prior with respect to the likelihood as specified in our algorithm. A sample from the density in the likelihood is accepted if it lies within the considered domain, and rejected if it falls outside the domain. This is so because of the assumption that the neural network has support strictly within the domain. Here we can see that the quality of the likelihood is a major factor in the success of the method. The dip in the acceptance rate can be found to negatively affect the mass of the neural network prior (Figure~\ref{fig:1d_oulin_graphserr}~(c)) and finally results in a spike in the error (Figure~\ref{fig:1d_oulin_graphserr}~(a)). Furthermore, it is noteworthy that the method seems to recover from this event after the next two time steps which is a further hint at the stability of our method. 

\begin{figure}
     \centering
     \begin{subfigure}[b]{0.48\textwidth}
         \centering
         \includegraphics[width=\textwidth]{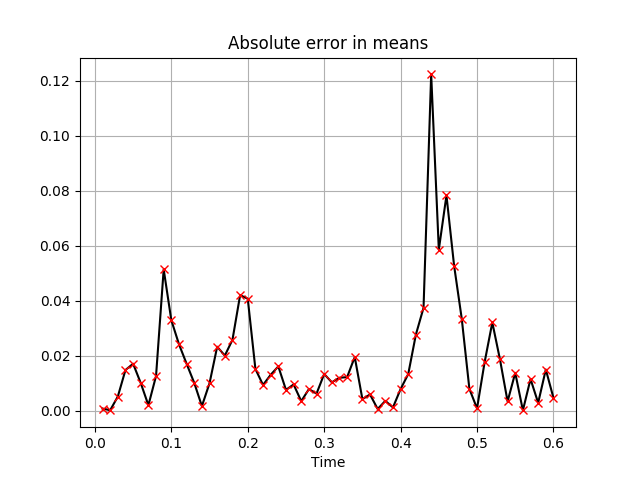}
         \caption{}
         \label{fig:y56367547srewfewrt x}
     \end{subfigure}
     \hfill
     \begin{subfigure}[b]{0.48\textwidth}
         \centering
         \includegraphics[width=\textwidth]{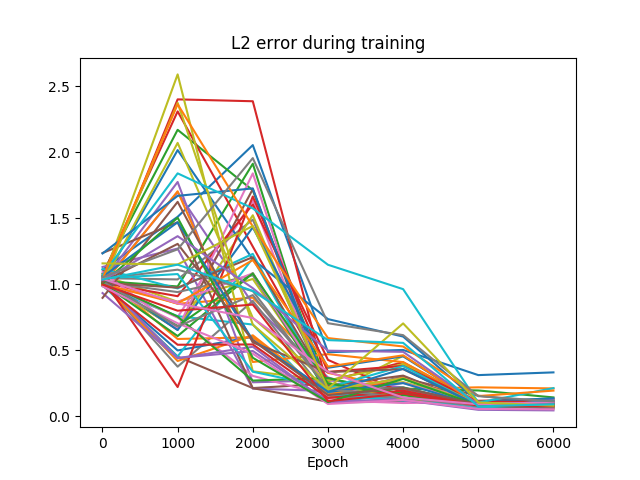}
         \caption{}
         \label{fig:5345633twefgrg}
     \end{subfigure}
     \newline
     \begin{subfigure}[b]{0.48\textwidth}
         \centering
         \includegraphics[width=\textwidth]{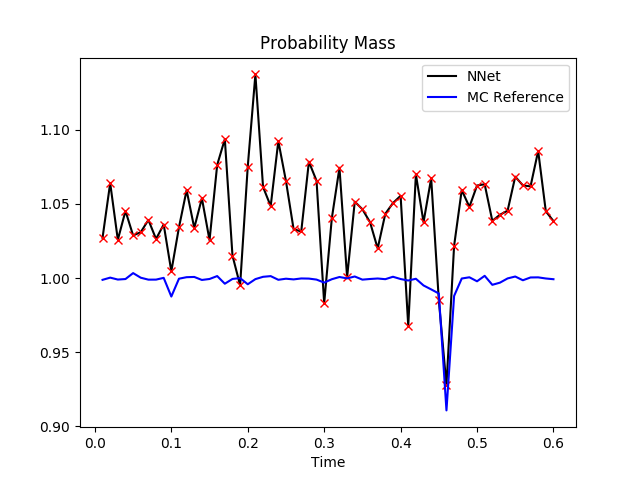}
         \caption{}
         \label{fig:5435346xtegxvwfrews}
     \end{subfigure}
     \hfill
     \begin{subfigure}[b]{0.48\textwidth}
         \centering
         \includegraphics[width=\textwidth]{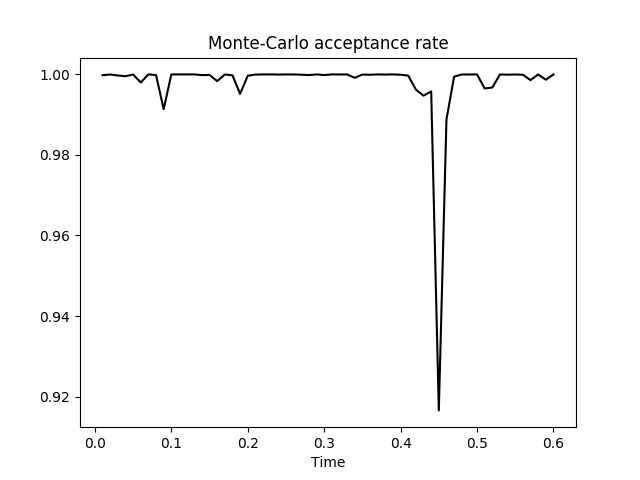}
         \caption{}
         \label{fig:t54643trvre42f}
     \end{subfigure}
        \caption{Error and diagnostics for linear filter, case 1. (a) Absolute error in means between the approximated distribution and the exact solution. (b) $L_2$ error of the neural network during training with respect to the Monte-Carlo reference solution. (c) Probability mass of the neural network prior. (d) Monte-Carlo acceptance rate.}
        \label{fig:1d_oulin_graphserr}
\end{figure}

\subsubsection{Linear filter, case 2: \textit{M} = 1, \textit{\texteta} = --1}\label{sec:lin_num} 
The second numerical study of this work is based on the Kalman filtering setting with the set of parameters given in Table~\ref{tab:1d_lin}. 

\begin{table}[ht]
\centering
\begin{tabular}{ c c c c c c c c }
$x_0$ & $y_0$ & $M$ & $\eta$ & $\Sigma$ & $H$ & $\gamma$ & $\Delta t$   \\
\hline\\[-1em]
$0.0$ & $0.0$ & $1.0$ & $-1.0$ & $0.1$ & $90.0$ & $0.0$ & $0.01$\\
\end{tabular}
\caption{Parameters used in the numerical experiment for the one-dimensional linear filter, case 2.}
\label{tab:1d_lin}
\end{table}

As the initial density we chose a Gaussian density with mean $0.0$ and standard deviation $0.01$. The domain over which we resolve the solution was chosen as the interval $[-0.8,0.4]$, in anticipation of the drift of the signal. We again iterated our method over $60$ time steps up to a final time of $0.6$. The results of the simulation are shown in Figure~\ref{fig:1d_lin_graphs}.

\begin{figure}[!ht]
  \centering
  \parbox{\figrasterwd}{
    \parbox{.55\figrasterwd}{%
      \subcaptionbox{}{\includegraphics[width=1.2\hsize]{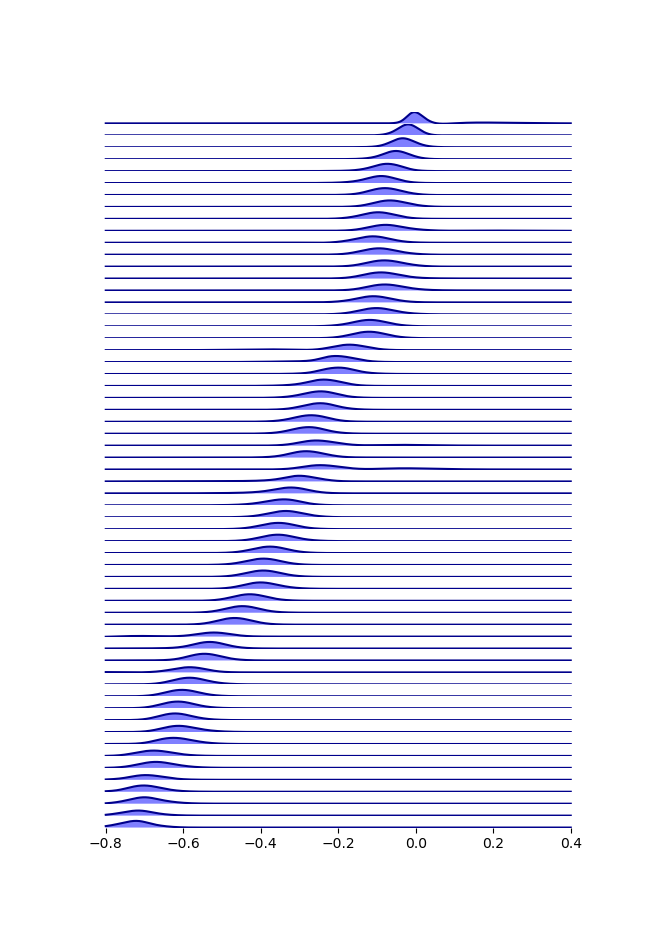}}
    }
    \hskip1em
    \parbox{.35\figrasterwd}{%
      \subcaptionbox{}{\includegraphics[width=\hsize]{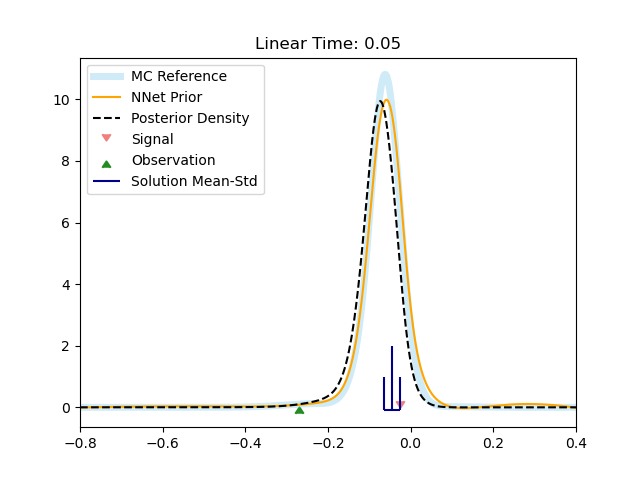}}
      \vskip1em
      \subcaptionbox{}{\includegraphics[width=\hsize]{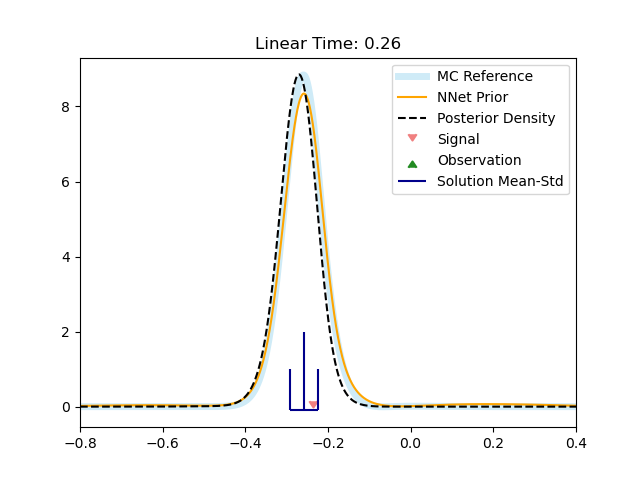}}
      \vskip1em
      \subcaptionbox{}{\includegraphics[width=\hsize]{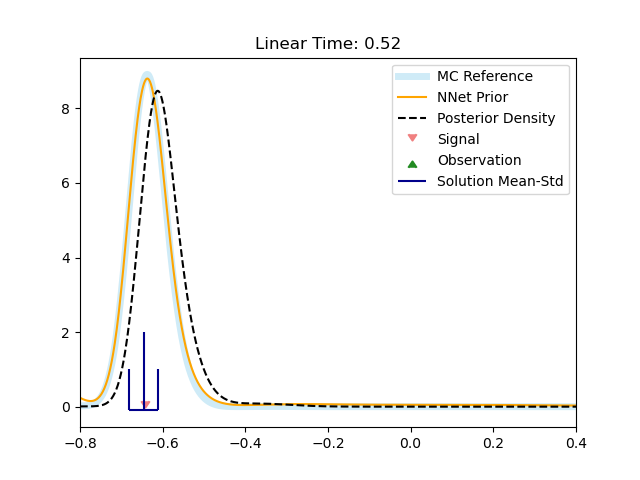}} 
    }
  }
  \caption{Results of the combined splitting-up/machine-learning approximation applied iteratively to the linear filtering problem, case 2. (a) The full evolution of the estimated posterior distribution produced by our method, plotted at all intermediate timesteps. (b)-(d) Snapshots of the approximation at an early time, $t=0.05$, an intermediate time, $t=0.26$, and a late time, $t=0.52$, obtained after 5, 26 and 52 iterations of our method, respectively. The black dotted line in each graph shows the estimated posterior, the yellow line the prior estimate represented by the neural network, and the light-blue shaded line shows the Monte-Carlo reference solution for the Fokker-Planck equation.}
\label{fig:1d_lin_graphs}
\end{figure}

As expected, the mean of the posterior moves to the left by approximately $0.01$ units of the domain at each time step. Furthermore, the standard deviation also initially increases as time progresses. 

\begin{figure}
     \centering
     \begin{subfigure}[b]{0.48\textwidth}
         \centering
         \includegraphics[width=\textwidth]{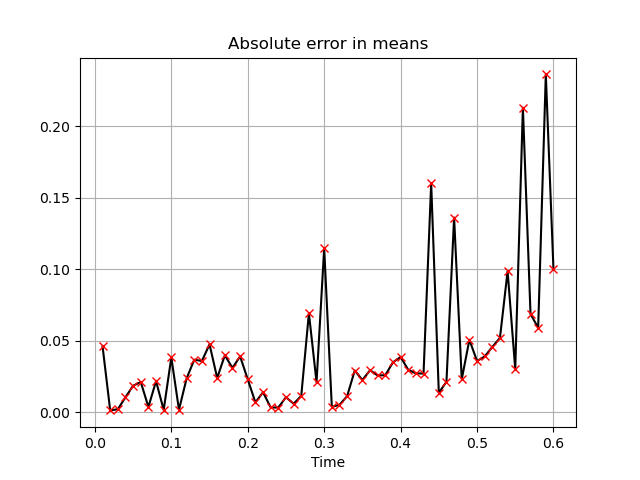}
         \caption{}
         \label{fig:y56367547srewrt x}
     \end{subfigure}
     \hfill
     \begin{subfigure}[b]{0.48\textwidth}
         \centering
         \includegraphics[width=\textwidth]{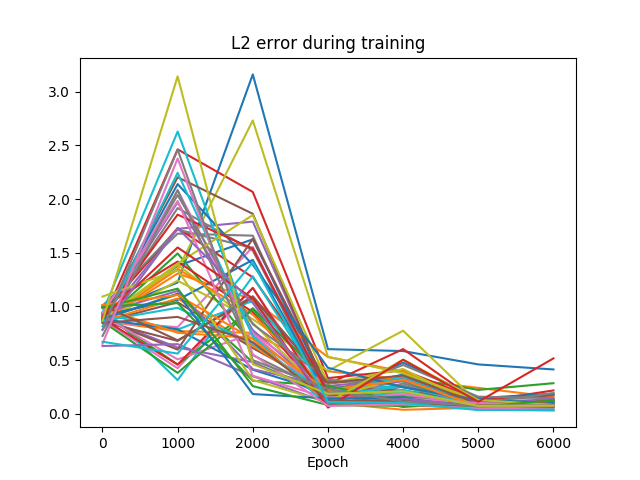}
         \caption{}
         \label{fig:5345633twerg}
     \end{subfigure}
     \newline
     \begin{subfigure}[b]{0.48\textwidth}
         \centering
         \includegraphics[width=\textwidth]{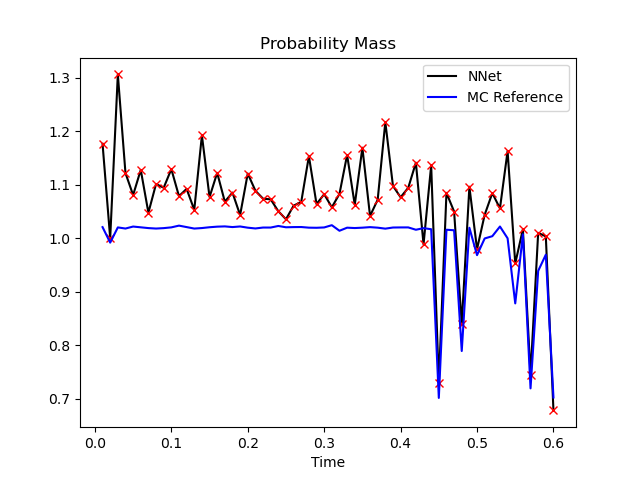}
         \caption{}
         \label{fig:5435346xtewfrews}
     \end{subfigure}
     \hfill
     \begin{subfigure}[b]{0.48\textwidth}
         \centering
         \includegraphics[width=\textwidth]{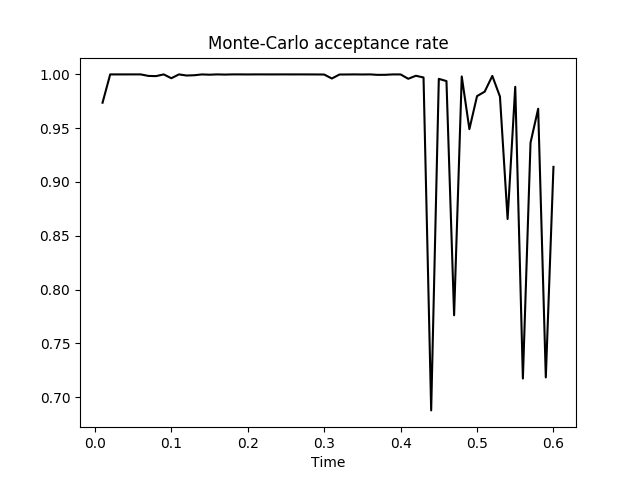}
         \caption{}
         \label{fig:t54643tr42f}
     \end{subfigure}
        \caption{Error and diagnostics for linear filter, case 2. (a) Absolute error in means between the approximated distribution and the exact solution. (b) $L_2$ error of the neural network during training with respect to the Monte-Carlo reference solution. (c) Probability mass of the neural network prior. (d) Monte-Carlo acceptance rate.}
        \label{fig:1d_lin_graphserr}
\end{figure}

In Figure~\ref{fig:1d_lin_graphserr} (a) we show the error between the means of the approximate posterior and the mean of the exact solution of the linear filter. Up to the time of about 0.44, we observe a steady oscillation within a range of 0.00-0.05, except for a few spikes which are classified as outliers. Thereafter, the error increases systematically. This phenomenon coincides with the observation in Figure~\ref{fig:1d_lin_graphserr} (c) where, after the time of about 0.44 the total mass of the network prior becomes unstable. Before this time, the neural network model has slightly overestimated the mass of the solution of the Fokker-Planck equation. Figure~\ref{fig:1d_lin_graphserr} (d) provides the interpretation for the cause of this phenomenon. It shows the Monte-Carlo acceptance rates for the integration method of the neural network prior with respect to the density given by the likelihood. The drop in acceptance rate shows that the samples from the likelihood increasingly lie outside the domain of the neural network prior, which depletes the quality of the approximation. Therefore, a strong likelihood within the domain we are considering is an important factor in the performance of our algorithm. This observations is also connected to the so-called signal-to-noise ratio which we need to be high in order to perform an accurate normalisation using the sampling method.
Finally, Figure~\ref{fig:1d_lin_graphserr} (b) is an illustration of the neural network training progress. Each line in the plot corresponds to a timestep, and shows the $L_2$ error against the training epoch with respect to the Monte-Carlo reference solution of the Fokker-Planck equation.

\subsection{One-dimensional Benes filter}\label{sec:1d-benes_res}
The third numerical study of this work is based on the nonlinear Benes filtering setting outlined in Example~\ref{ex:benes-filter}.
Here, we are considering the set of parameters, corresponding to the notation in Example~\ref{ex:benes-filter}, given in Table~\ref{tab:1d_ben}.

\begin{table}[ht]
\centering
\begin{tabular}{ c c c c c c c c }
$x_0$ & $y_0$ & $\alpha$ & $\beta$ & $\sigma$ & $h_1$ & $h_2$ & $\Delta t$   \\
\hline\\[-1em]
$0.0$ & $0.0$ & $3.0$ & $0.0$ & $0.5$ & $3.0$ & $0.0$ & $0.1$\\
\end{tabular}
\caption{Parameters used in the numerical experiment for the one-dimensional Benes filter.}
\label{tab:1d_ben}
\end{table}

The initial condition was again chosen to be the Gaussian density with mean $0.0$ and standard deviation $0.01$. This time, however, we chose a different, larger, time step in order to observe the characteristic bimodality appearing in the solution of the Benes filter. This also necessitated the choice of a larger domain for the neural net, which here was chosen to be the interval $[-4.0, 4.0]$. The results were calculated by iterating our scheme over $12$ timesteps for the approximation of the Benes filter and are plotted in Figure~\ref{fig:1d_ben_graphs}. The feature we like to stress in this nonlinear example is the development of the bimodal density that is resolved by our method, in particular in Figure~\ref{fig:1d_ben_graphs}~(c) and (d).

\begin{figure}[!ht]
  \centering
  \parbox{\figrasterwd}{
    \parbox{.55\figrasterwd}{%
      \subcaptionbox{}{\includegraphics[width=1.2\hsize]{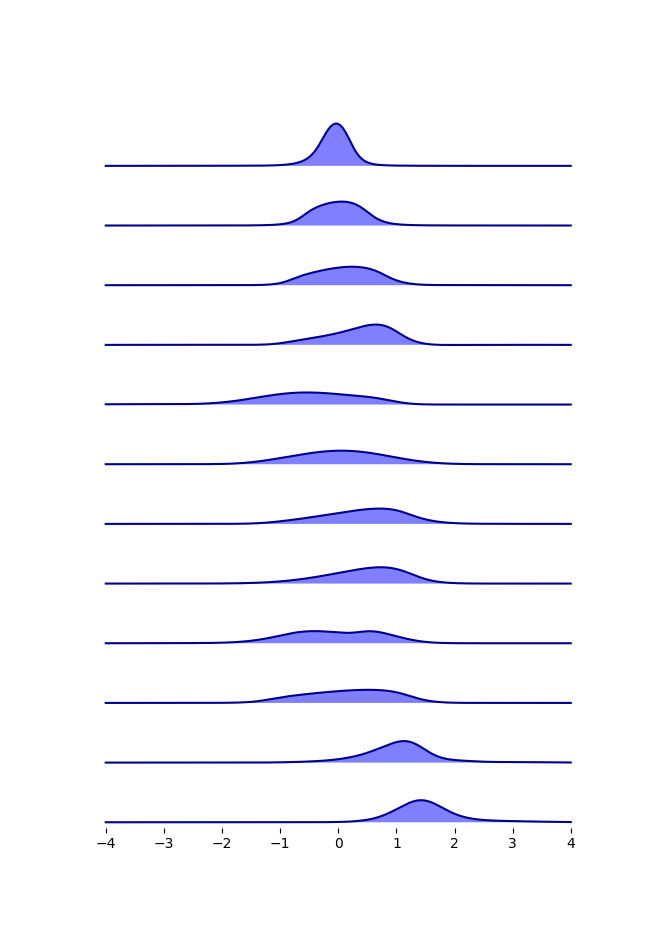}}
    }
    \hskip1em
    \parbox{.35\figrasterwd}{%
      \subcaptionbox{}{\includegraphics[width=\hsize]{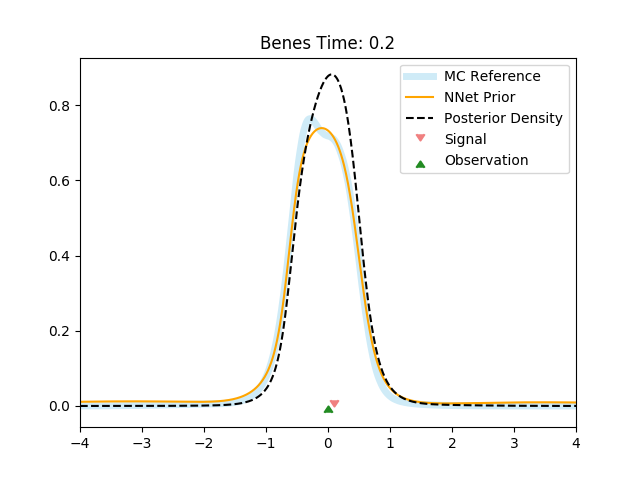}}
      \vskip1em
      \subcaptionbox{}{\includegraphics[width=\hsize]{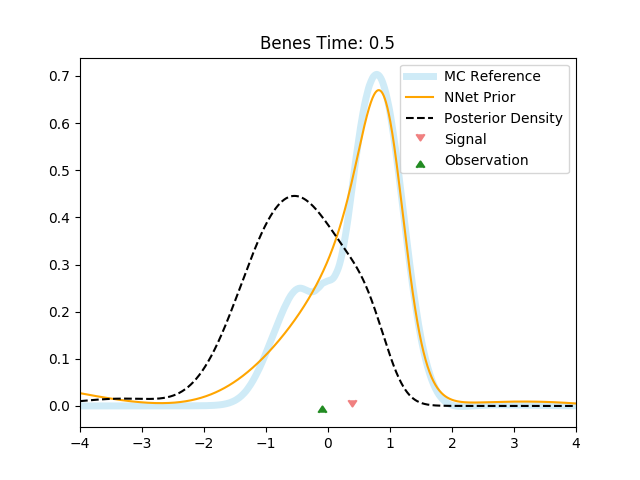}}
      \vskip1em
      \subcaptionbox{}{\includegraphics[width=\hsize]{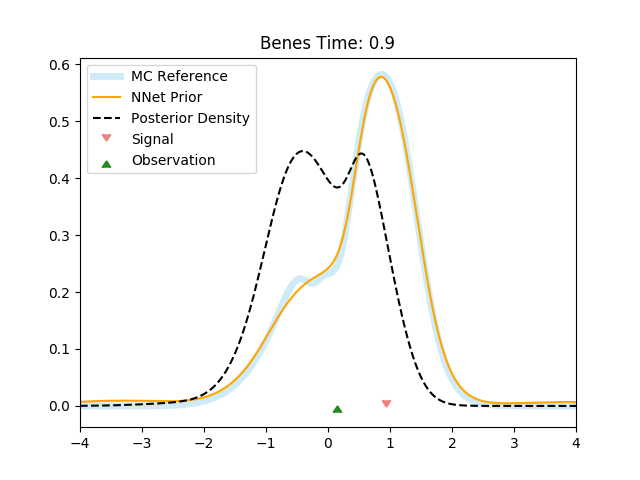}} 
    }
  }
  \caption{Results of the combined splitting-up/machine-learning approximation applied iteratively to the Benes filtering problem. (a) The full evolution of the estimated posterior distribution produced by our method, plotted at all intermediate timesteps. (b)-(d) Snapshots of the approximation at an early time, $t=0.2$, an intermediate time, $t=0.5$, and a late time, $t=0.9$, obtained after 2, 5 and 9 iterations of our method, respectively. The black dotted line in each graph shows the estimated posterior, the yellow line the prior estimate represented by the neural network, and the light-blue shaded line shows the Monte-Carlo reference solution for the Fokker-Planck equation.}
\label{fig:1d_ben_graphs}
\end{figure}

The error and diagnostic plots are shown in Figure~\ref{fig:1d_ben_graphserr}. The absolute error in Figure~\ref{fig:1d_ben_graphserr}~(a) shows a steady oscillation, and Figure~\ref{fig:1d_ben_graphserr}~(b) indicates that the neural network training converges to the Monte-Carlo reference solution across all time steps. Moreover, the probability mass plotted in Figure~\ref{fig:1d_ben_graphserr}~(c) oscillates around the correct value $1.0$ with a slight tendency to underestimate, also for the Monte-Carlo reference. The initially low mass is explained by the sharp drop of the peak of the initial Gaussian during the first timestep, which is difficult to capture. As observed in the linear cause though, the method seems to be able to recover from occasional inaccuracies. Figure~\ref{fig:1d_ben_graphserr}~(d) shows the Monte-Carlo acceptance rate for the correction step. The final drop is still acceptable, as the value of $\sim 93\%$ acceptance rate is still reasonable. These results demonstrate an ability of our algorithm to also track nonlinear problems over several timesteps.

\begin{figure}
     \centering
     \begin{subfigure}[b]{0.48\textwidth}
         \centering
         \includegraphics[width=\textwidth]{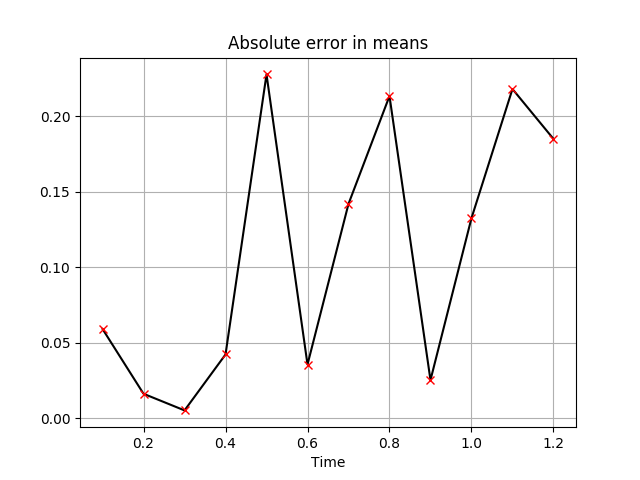}
         \caption{}
         \label{fig:y56367547srewgdsrt x}
     \end{subfigure}
     \hfill
     \begin{subfigure}[b]{0.48\textwidth}
         \centering
         \includegraphics[width=\textwidth]{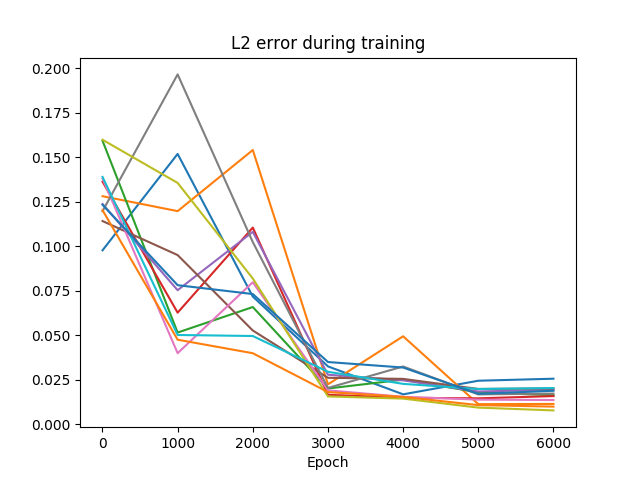}
         \caption{}
         \label{fig:5345633gdftwerg}
     \end{subfigure}
     \newline
     \begin{subfigure}[b]{0.48\textwidth}
         \centering
         \includegraphics[width=\textwidth]{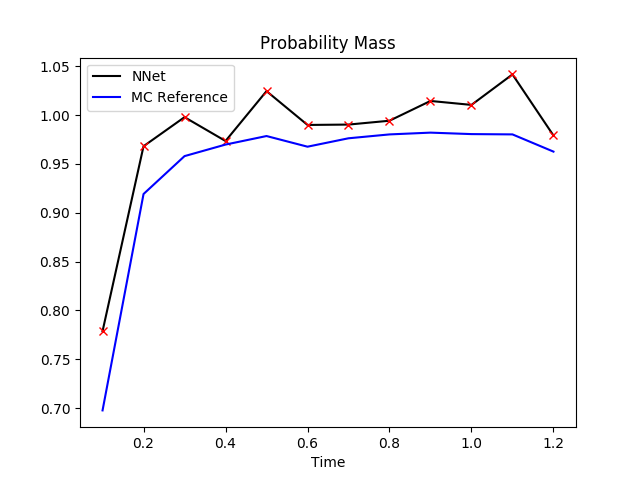}
         \caption{}
         \label{fig:5435346xhtyrtewfrews}
     \end{subfigure}
     \hfill
     \begin{subfigure}[b]{0.48\textwidth}
         \centering
         \includegraphics[width=\textwidth]{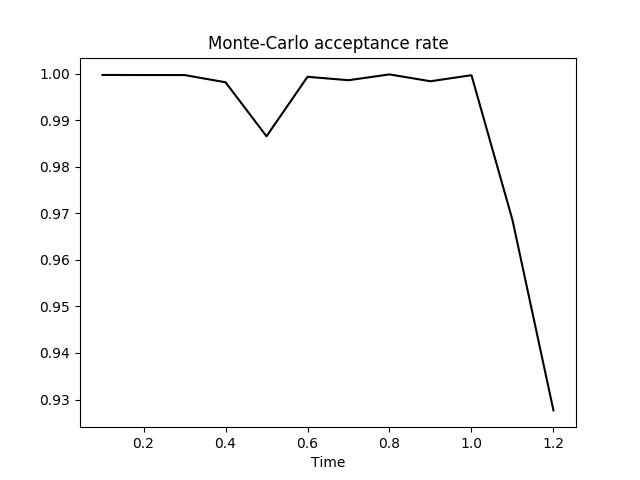}
         \caption{}
         \label{fig:t54643trhytrh42f}
     \end{subfigure}
        \caption{Error and diagnostics for the Benes filter. (a) Absolute error in means between the approximated distribution and the exact solution. (b) $L_2$ error of the neural network during training with respect to the Monte-Carlo reference solution. (c) Probability mass of the neural network prior. (d) Monte-Carlo acceptance rate.}
        \label{fig:1d_ben_graphserr}
\end{figure}

\section{Conclusion and outlook}\label{sec:conclusion}
As observed, an important factor in the success of our method lies in accurately determining the domain of resolution \emph{before} beginning the iterative procedure. As the mass of the density begins to move outside our observed window, the results will degrade quickly. A possible solution is to shift the observed window in a suitable manner at regular time intervals to obtain an adaptive method. Moreover, due to the Monte-Carlo sampling based correction step, which relies on samples from the likelihood, we need a high signal-to-noise ratio to maintain an accurate evaluation of the integral in the domain. If the acceptance rate of Monte-Carlo samples from the likelihood drops significantly, the results in our method deteriorate. This can be counteracted by sampling more points from the distribution. However, if the likelihood spread is too large, this will significantly slow down the algorithm.

We do not think that dealing with the domain boundaries is an unsurmountable problem. Future research will focus on investigating approaches to deal with the motion of the posterior outside of the domain.

Note further that, because the density of the optimal filter changes continuously in time, our algorithm is a natural candidate for \emph{transfer learning} the parameters of the neural net instead of retraining them from a random initialisation at every time step. Further details on the area of transfer learning can be found, for example, in~\cite[Chapter 15.2]{goodfellow2016deep}.

Although we found a similar performance of our method across a range of different hyperparameters such as the batch size, the network architecture, etc. the optimal choice of these for our given problem of filtering remains open.

A further direction of future study will be a detailed error analysis of the presented algorithm. This is a complex problem because the approximation performed here introduces inaccuracies at various stages. The first ones are the usual simulations of the signal and observation processes, as well as now also the auxiliary diffusion. Moreover, the machine learning algorithm introduces an error in estimating the Fokker-Planck PDE solution. Finally, the error due to the Monte-Carlo normalisation in the correction step must be analysed.

\subsubsection*{Conflict of Interest}
\vspace{-2mm}
 The authors declare that they have no conflict of interest.
\vspace{-2mm}

\subsubsection*{Data availability}
\vspace{-2mm}
 The document contains no data.
\setlength{\bibitemsep}{2.5 pt}

\printbibliography

\end{document}